\documentclass[journal]{IEEEtran}
\ifCLASSINFOpdf
\else
\fi
\usepackage{amsmath}
\usepackage{amssymb}
\usepackage{amsthm}
\usepackage{placeins}
\usepackage[export]{adjustbox}

\usepackage{braket}
\usepackage{siunitx}
\usepackage{url}
\usepackage[colorlinks = true,
            linkcolor = black,
            urlcolor  = black,
            citecolor = black,
            anchorcolor = black]{hyperref}
\usepackage{enumitem}
\usepackage[table,xcdraw]{xcolor}
\usepackage{tikz}
\usetikzlibrary{shapes,arrows.meta, arrows,quotes,positioning,tikzmark,fit,backgrounds}
\tikzstyle{block} = [rectangle, draw, text centered, rounded corners]
\tikzstyle{box} = [rectangle, draw,text width=5em,text centered]
\tikzstyle{line} = [draw, -latex']

\usetikzlibrary{shapes.geometric, arrows}

\usepackage{tkz-graph}
\usepackage{breqn}
\usepackage{alphalph,etoolbox}
\patchcmd{\subequations}{\alph{equation}}{\alphalph{\value{equation}}}{}{}

\newcommand{\bsubeq}{\begin{subequations}}
\newcommand{\esubeq}{\end{subequations}}
\newcommand{\BI}{\begin{itemize}}
\newcommand{\EI}{\end{itemize}}
\newcommand{\I}{\item}

\newcommand{\cblue}{\color{black}}
\newcommand{\ccblue}{\color{black}}

\newcommand{\mc}{\mathcal}
\newcommand{\mb}{\mathbf}

\def\st{{\rm s.t.}}

\usepackage{algorithmic}
\usepackage{algorithm}
\usepackage{graphicx,amssymb}
\usepackage{enumerate}

\usepackage[normalem]{ulem}

\newcommand{\be}{\begin{enumerate}}
\newcommand{\ee}{\end{enumerate}}

\newcommand{\Gt}{\mc{G}^{\rm{Ther}}}
\newcommand{\Gw}{\mc{G}^{\rm{Wind}}}

\newcommand{\Tda}{\mc{T}^{\rm{DA}}}
\newcommand{\Trt}{\mc{T}^{\rm{RT}}}
\newcommand{\Trth}{\mc{T}^{\rm{RT|Hour}}}
\newcommand{\Nint}{N^{\rm{tp}}}
\renewcommand{\L}{\mc{L}}
\newcommand{\G}{\mc{G}}
\newcommand{\N}{\mc{N}}
\newcommand{\cstart}{C^{\rm{Start}}_g}
\newcommand{\cdown}{C^{\rm{Down}}_g}
\newcommand{\cvoll}{C^{\rm{VOLL}}}
\newcommand{\cvollrt}{C^{\rm{VOLL|RT}}}
\newcommand{\dbar}{\bar{D}_{it}}
\newcommand{\drt}{d^{\rm{RT}}_{it}}
\newcommand{\vdrt}{\bf{d}^{\rm{RT}}}
\newcommand{\drthat}{d^{\rm{RT}*}_{it}}
\newcommand{\punmet}{p^{\rm{Unmet}}_{it}}
\newcommand{\pcurtail}{p^{\rm{Curtail}}_{gt}}
\newcommand{\ptrans}{P^{\rm{Trans}}_{ij}}
\newcommand{\prt}{p^{\max,\rm{RT}}_{gt}}
\newcommand{\vprt}{\bf{p}^{\max,\rm{RT}}}

\newcommand{\Pmaster}{\mc{P}^{\rm{Master}}}

\newcommand{\pmaxbar}{\bar{P}_{gt}^{\max}}
\newcommand{\pmin}{P_{g}^{\min}}

\newcommand{\valpha}{\pmb \alpha}
\newcommand{\alphad}{\alpha^d_{kt}}
\newcommand{\alphaw}{\alpha^w_{kt}}
\newcommand{\alphadtil}{\tilde{\alpha}^d_{kt}}
\newcommand{\alphawtil}{\tilde{\alpha}^w_{kt}}

\newcommand{\vqd}{\mathbf{Q}^d_{kt}}
\newcommand{\vqw}{\mathbf{Q}^w_{kt}}

\newcommand{\vxrt}{\mathbf{x}^{\rm{RT}}}
\newcommand{\vomega}{\pmb \omega}
\newcommand{\vomegatil}{\tilde{\pmb \omega}}

\newcommand{\Xrt}{X^{\rm{RT}}}
\newcommand{\vy}{\mathbf{y}}
\newcommand{\vystar}{\mathbf{y}^*}
\newcommand{\ystar}{y^*_{gt}}

\newcommand{\uhat}{\hat{V}^*}
\newcommand{\lambdahat}{\lambda^*_{it}}
\newcommand{\Ione}{\mc{I}_{1t}}
\newcommand{\Izero}{\mc{I}_{0t}}
\newcommand{\Ionep}{\mc{I}'_{1t}}
\newcommand{\Izerop}{\mc{I}'_{0t}}

\renewcommand{\O}{\mc{O}}

\newcommand{\tstar}{t}
\newcommand{\dstarbar}{\bar{D}_{it}}
\newcommand{\lambdayo}{\lambda_{i\tstar}(\vy, \vomega)}

\newcommand{\lambdaystaromegastar}{\lambda_{i\tstar}(\vy^*, \vomega^*)}
\newcommand{\lambdart}{\lambda_{it}(\vy^*, \vomega)}
\newcommand{\dcopf}{\text{DCOPF}(\vystar, \vomega)}
\newcommand{\dcopftil}{\text{DCOPF}(\vystar, \vomegatil)}
\newcommand{\dcopfd}{\text{DCOPF-D}(\vystar, \vomega)}
\newcommand{\dcopfa}{\text{DCOPF-A}(\vystar)}
\newcommand{\dcopfADV}{\text{DCOPF-A}(\vystar)}

\newcommand{\Vhat}{\hat{V}}

\newcommand{\Mrt}{M^{\rm{RT}}_g}

\newcommand{\pcurtails}{p^{\rm{Curtail}}_{sgt}}
\newcommand{\var}{z}
\newcommand{\pmaxs}{P_{sgt}^{\max}}
\newcommand{\punmets}{p^{\rm{Unmet}}_{sit}}

\usepackage{longtable}
\usepackage{afterpage}
\newcommand{\sm}{\setminus}

\usepackage{mathtools}

\usepackage{multirow}
\usepackage{graphicx}
\usepackage{caption}
\usepackage{subfig}
\usepackage{breqn}
\allowdisplaybreaks

\usepackage[numbers,sort&compress]{natbib}
\usepackage{multicol}
\usepackage{hyperref}
\usepackage{algorithmic, algorithm}
\usepackage{amsmath, amsfonts, amssymb, amsthm}
\usepackage{enumitem}
\usepackage{bm}
\usepackage{cleveref}
\crefname{section}{§}{§§}
\Crefname{section}{§}{§§}
\newtheorem{theorem}{Theorem}[section]

\newtheorem{proposition}[theorem]{Proposition}

\usepackage{tikz}
\usetikzlibrary{positioning}

\usepackage{placeins}
\newcommand{\subparagraph}{}
\usepackage{titlesec}
\titlespacing*{\section}{0pt}{5pt}{3pt}
\titlespacing*{\subsection}{0pt}{3pt}{0pt}
\titlespacing*{\subsubsection}{0pt}{3pt}{0pt}
\begin{document}
%
\title{Risk-Aware Security-Constrained Unit Commitment{\cblue : Taming the Curse of Real-Time Volatility and Consumer Exposure}}
%
%
%

\author{
Daniel Bienstock, 
Yury Dvorkin, 
Cheng Guo, 
Robert Mieth, 
Jiayi Wang
\thanks{This work was supported by the US DOE ARPA-e under Grant DE-AR0001300.}
\thanks{D. Bienstock: 
Department of Industrial Engineering and Operations Research, Columbia University, New York, NY 10027 USA, 
\href{mailto:dano@columbia.edu}{dano@columbia.edu}
}
\thanks{Y. Dvorkin: 
Department of Electrical and Computer Engineering, Department of Civil and System Engineering, Ralph O'Connor Sustainable Energy Institute,  Johns Hopkins University, Baltimore, MD 21218 USA, 
\href{mailto:ydvorki1@jhu.edu}{ydvorki1@jhu.edu}
}
\thanks{C. Guo: 
School of Mathematical and Statistical Sciences, Clemson University, Clemson, SC 29634 USA, 
\href{mailto:cguo2@clemson.edu}{cguo2@clemson.edu}
}
\thanks{R. Mieth: 
Department of Industrial and Systems Engineering, Rutgers University, New Brunswick, NJ 08901 USA, 
\href{mailto:robert.mieth@rutgers.edu}{robert.mieth@rutgers.edu}
}
\thanks{J. Wang: 
Department of Management Science and Engineering, Stanford University, Stanford, CA 94305 USA, 
\href{mailto:jyw@stanford.edu}{jyw@stanford.edu}
}
}

\maketitle

\begin{abstract}
We propose an enhancement to wholesale electricity markets to contain the exposure of consumers to \textit{increasingly large and volatile consumer payments} arising as a byproduct of volatile real-time net loads -- i.e., loads minus renewable outputs -- and prices, both compared to day-ahead cleared values. We incorporate a trade-off, motivated by portfolio optimization methods, between standard day-ahead payments and a robust estimate of such excess payments into the day-ahead computation and specifically seek to account for \textit{volatility} in real-time net loads and renewable generation. \color{black}{Our model features a data-driven uncertainty set based on principal component analysis, which accommodates both load and wind production volatility and captures locational correlation of uncertain data. To solve the model more efficiently, we develop a decomposition algorithm that can handle nonconvex subproblems. Our extensive experiments on {\cblue a realistic NYISO data set} show that the risk-aware model protects the {\cblue consumers} from potential high costs caused by adverse circumstances.}
\end{abstract} 


%
\IEEEpeerreviewmaketitle

\section{Introduction}\label{sec: intro}
{\cblue Wholesale electricity markets are typically organized in two stages. First, day-ahead (DA) markets employ Security Constrained Unit Commitment (SCUC) computations  \cite{al2014role} to secure generator commitments as well as generation and load amounts for the forthcoming day, which are paid for using locational marginal prices (LMPs).  Second, given the DA outcomes and updated forecast of loads and renewables and availability of generation and transmission resources, the real-time (RT)  markets identify additional generation amounts to match actual RT loads, which are additionally paid for using RT LMPs\footnote{\cblue The DA schedule is updated via the Reliability Unit Commitment (RUC) computation \citep{helman2008design}; however, this computation has no direct financial impact.}.  

Any RT load not cleared in the DA stage is paid for using RT LMPs; we term such obligations \textit{RT consumer exposure}. Figure \ref{ecp} details possible discrepancies between DA and RT loads, which are a source of financial risk especially when combined with high RT LMPs.  These discrepancies  are also exacerbated by \textit{volatility} (i.e., variation around the mean) in loads and environmental conditions, an increasingly important factor in high-renewable markets. 
\begin{figure}[b!]
  \centering
     \includegraphics[height=1.85in]{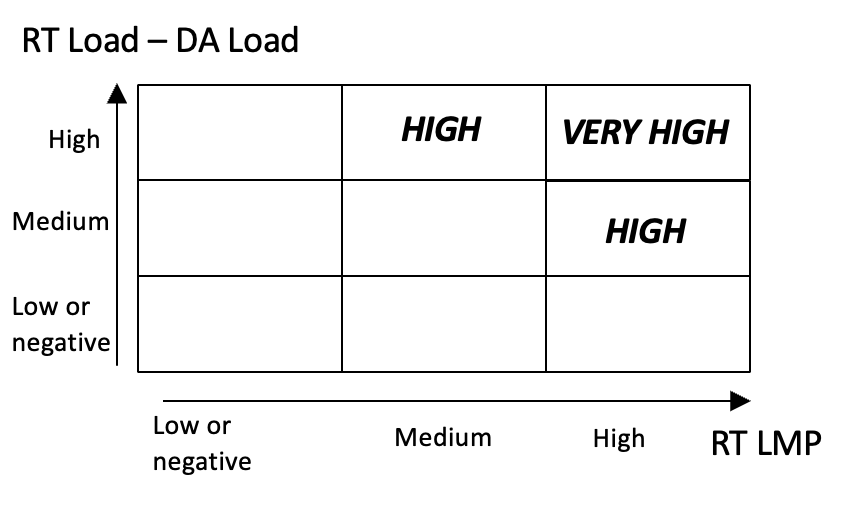} 
	\caption{RT consumer exposure risk profile} 
	\label{ecp}  
\end{figure}
We further note that the DA market carries a financial obligation for both loads and generators; namely, loads pay at the DA LMP -- in other words the current DA computation is not set up to handle the risk of high RT consumer exposure.

This paper addresses the tradeoff between the risk represented by high RT consumer exposure and the traditional cost-minimization (or welfare maximization) provided by DA SCUC. We remark that the DA SCUC (and RUC) computation is deterministic and currently uses single-point estimates of load averages for each specific hourly or half-hourly period. RT dispatch, on the other hand, relies on estimates of \textit{average} loads in the ensuing time window, typically spanning 5 minutes.} Significant load or generation deviations within the RT window are handled via reserves, which are set up as exogenous reserve requirements. This combination of prediction-driven scheduling and {\cblue RT} correction has proved successful {\cblue in low-renewable markets} and is a strong engineering accomplishment {\cblue \cite{national2021future}}.

{\cblue Ongoing} large-scale deployment of renewable generation, as well as battery resources and controllable loads, challenges this paradigm. In particular, renewables may, under adverse circumstances, introduce large and correlated {\cblue RT} deviations from expected generation levels, thus increasing \textit{volatility}. It is worth noting that even high-quality forecasts cannot overcome   volatility, which is a reflection of {\cblue RT} stochasticity;
in other words uncertainty at the 5-minute level may be present even  generally accurate estimates of \textit{average} renewable output or load are available at the {\cblue DA} phase.
{\cblue The RT settlement mechanism implies that the consumers (and public, by extension) are ultimately responsible for volatility-derived costs, in the form of high {\cblue RT consumer exposure} -- partially or fully offsetting the economic advantage of zero-marginal renewable generation.  High excess consumer 
payments, if frequently observed, could become of interest to regulatory and government entities, and may potentially reduce public appetite for adopting renewable generation \cite{murray2019paradox}.} 


{\cblue This paper presents a volatility-aware algorithmic enhancement to the DA SCUC and RUC computations, with the goal of yielding DA decisions that effectively trade-off DA costs against {\cblue RT consumer exposure} arising from volatility.  A second goal in our work is to develop a DA SCUC alternative that minimally changes the mechanics of current wholesale electricity markets and the overall operation of power systems, including reserves and balancing.  Moreover, we do not alter the {\cblue RT} markets.}

We are specifically interested in financial and physical risks raised by \textit{critical time periods} characterized by sharp load and/or renewable volatility. A pertinent example is provided by peak hours on high-load, high-variability or otherwise stressed days, which, as a result may experience very high discrepancy between {\cblue RT} and {\cblue DA} prices. See, for example,  Figure 17 of \cite{zarnikau2015day} and Figure 3-7 of \cite{ISONE2022AnnualMarkets}. Empirical evidence in \citep{CAISO_OASIS, jha2013testing} also suggests that, additionally, virtual traders can forecast, and thus take advantage of, these periods quite accurately. This point serves to underscore the magnitude of the financial discrepancies between {\cblue DA} and {\cblue RT} markets.

{\cblue We stress that current DA SCUC routines used in practice are not immune to the {\cblue RT consumer exposure} under RT LMP spikes that as we discussed are due to excess RT load and shortfall in wind production. This is a major concern as the still largely inflexible consumers end up paying more, especially under increasing penetration of renewable energy and under extreme events \cite{levin2022extreme, busby2021cascading}.} Our approach explicitly internalizes RT price formation under adversarial conditions in the DA UC process and, using a realistic New York Independent System Operator (NYISO) test system {\cblue and demonstrates its benefits to reducing the overall RT consumer exposure}. 

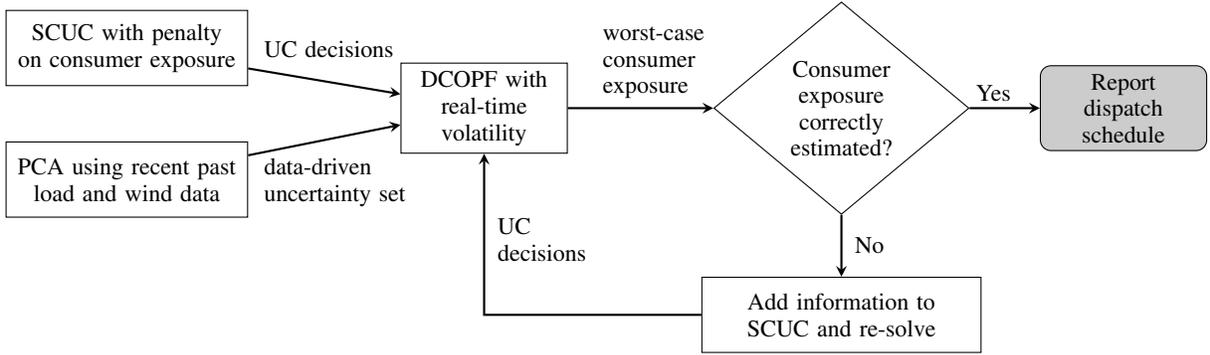
\begin{figure*}[ht]
\begin{center}
\vspace*{-0.4cm}
{
\begin{tikzpicture}[node distance=1.75cm]
\tikzstyle{startstop} = [rectangle, rounded corners, minimum width=1.15cm, minimum height=0.7cm,text centered, draw=black, fill = gray!40!white]
\tikzstyle{process} = [rectangle, minimum width=2cm, minimum height=1cm, text centered, text width=1.98cm, draw=black]
\tikzstyle{decision} = [diamond, aspect=1.4, text centered, text width=1.2cm, draw=black]
\tikzstyle{arrow} = [thick,->,>=stealth]

\small
\node (scuc) [process,xshift = 0.7cm,text width=3cm] {SCUC with penalty on consumer exposure};
\node (pca) [process,below of = scuc,text width=3cm] {PCA using recent past load and wind data};
\node (dcopf) [process, right of=scuc,xshift = 3cm, yshift = -0.8cm, text width=2cm] {DCOPF with real-time volatility};


\node (expo) [decision,right of=dcopf,aspect=1.2, xshift=3cm, text width=1.4cm] {Consumer exposure correctly estimated?};
\node (stop) [startstop, right of=expo,xshift = 2cm,text width=2cm] {Report dispatch schedule};
\node (resolve) [process, aspect=1.33,below of = expo, text width=3.5cm, yshift=-1cm]  {Add information to SCUC and re-solve};
\node (dummy3) [left of=resolve] at (dcopf |- resolve) {};


\draw [arrow] (scuc) -- node[anchor=south,xshift = 0.2cm,yshift=0.2cm, text width = 2cm]{UC decisions} (dcopf);
\draw [arrow] (pca) -- node[anchor=south,xshift = 0.2cm,yshift=-1cm, text width = 2cm]{data-driven uncertainty set} (dcopf);
\draw [arrow] (expo) -- node[anchor=east,xshift = 0.2cm,yshift=0.2cm] {Yes} (stop);
\draw [arrow] (expo) -- node[anchor=east, xshift = 3cm,yshift=0.01cm, text width = 2.7cm] {No} (resolve);

\draw [arrow] (resolve) -| node[anchor=east,xshift = 1.8cm,yshift=1cm, text width=1.5cm] {UC decisions} (dcopf);
\draw [arrow] (dcopf) -- node[anchor=south,xshift = 0.5cm,yshift=0.01cm, text width = 2cm]{worst-case consumer exposure} (expo);
\end{tikzpicture}
}
\caption{\cblue Our method obtains a risk-aware day-ahead dispatch schedule, by iteratively updating the SCUC model with information on worst-case real-time consumer exposure under different UC decisions.}
\label{fig: model_flowchart}
\end{center}
\end{figure*}


\subsection{Literature Review}\label{ch: lit_review}
The {\cblue wide-spread} deployment of {\cblue uncertain renewable generation} has previously motivated numerous proposals to improve upon currently static reserve requirements and deterministic point forecasts through alternative probabilistic models \cite{zheng2014stochastic,roald2023power}.
{\cblue From a research perspective, the SCUC computation is commonly recast as} a two-stage stochastic program (SP) 
{\cblue so as} to tackle demand uncertainty \cite{takriti1996stochastic,carpentier1996stochastic} and {\cblue variable} renewable generation \cite{morales2009economic,wu2007stochastic,dvorkin2014hybrid,wang2015real,sundar2016unit}. 
These approaches optimize 
{\cblue the sum of} UC cost (the first stage) {\cblue plus expected} dispatch (second stage), and model uncertainty through scenarios or an estimated probability distribution.
A popular alternative to SP-based approaches, avoiding their often prohibitive computational complexity and scenario requirements \cite{wu2007stochastic}, is adaptive robust optimization (ARO)  \cite{jiang2011robust,bertsimas2012adaptive}. This approach is a variant of robust optimization \citep{bertsimas2004price}. ARO models minimize {\cblue DA} UC and {\cblue RT} dispatch cost for a worst-case, {\cblue RT} realization drawn from a pre-defined uncertainty set.
For a given uncertainty set, this approach is computationally efficient and does not require {\cblue an} estimation or assumption of a specific underlying probability distribution.
The uncertainty set itself can be estimated from data, but requires careful tuning to achieve good quality solutions that are not overly conservative \cite{golestaneh2018polyhedral}.
Polyhedral sets as in \citep{jiang2011robust, bertsimas2012adaptive} are easy to implement but do not capture correlation information. {\cblue Data-driven methods have recently been used to create uncertainty sets. For instance, a Dirichlet process mixture model is used in \cite{ning2019data} to construct data-driven uncertainty sets for wind forecast errors, while convex combination of historical renewable generation profiles are used for the data-driven uncertainty sets in \cite{velloso2019two}. These uncertainty sets are shown to capture more complex dynamics of the uncertain data and avoid overestimation of variations.}

{\cblue Data-driven ARO (DDARO) models for UC incorporate data-driven uncertainty sets in ARO models. This approach has several benefits such as better interpretability \cite{velloso2019two} and cost reduction \cite{ning2019data}. There are different types of objective functions used in DDARO UC models. For example, \cite{velloso2019two} minimizes the DA operation cost and worst-case total RT redispatch cost, while \cite{ning2019data,zhao2022sustainable} minimize total commitment and worst-case dispatch costs. To the best of our knowledge, no DDARO model (or ARO model in general) in the literature addresses the RT consumer exposure in the RT market. In addition, some models in the literature depicts a different market clearing process than the current practice \cite{ning2019data,zhao2022sustainable}, suggesting a need for a more significant reform of the existing two-stage market. Our proposed model aims to protect the consumers from RT price spikes, while keeping our setup very close to the actual DA market practice so as to ease implementation.

}

{\cblue We would like to note the difference between our work and the literature on consumer payment minimization. The SCUC computation in \cite{fernandez2013network} is modified to minimize a cost objective equal to the total consumer payment (rather than commitment plus generation costs) using deterministic {\cblue DA} estimates for loads.  A restructuring of the {\cblue RT} pricing mechanism that differs from the LMP setup is considered in \cite{wang2023optimization} with the goal of obtaining competitive equilibria and thus lowering consumer prices.  In contrast, the {\cblue RT consumer exposure} in our model is based on \textit{high} RT load  and RT price spikes, both as a result of {\cblue RT} volatility. Our goal is to produce a DA unit commitment that trades-off DA cost versus RT consumer exposure due to volatility -- while producing a market structure closely aligned with the current practice.} 

\subsection{Our contributions}\label{ch: our_conts}
 {\cblue This paper proposes} a two-stage risk-aware optimization model for the {\cblue DA} SCUC problem that explicitly models both {\cblue DA and RT} operations with a focus on {\cblue RT consumer exposure}.
We solve this model using a customized cutting plane-based algorithm. 
Inspired by factor stressing used in the financial services industry \citep{haughfactor}, we construct a data-driven uncertainty set based on principal component analysis (PCA) \citep{wold1987principal} to capture locational {\cblue correlations, which model} the stochasticity in both load and wind generation more realistically.
While the constructed uncertainty set is similar to \cite{ning2018data} (as used for UC in \cite{zhao2022sustainable}), it is different in structure and adapted for our proposed DA SCUC model. 
To demonstrate our model on a real-world, large-scale NYISO case study, we also develop heuristics to reformulate the uncertainty set. {\cblue Figure \ref{fig: model_flowchart} provides an overview of our modifications to a DA SCUC model.}


We summarize our contributions as follows:
\BI
\I Inspired by the risk perspective provided by the mean-variance portfolio optimization problem \citep{markowitz1952portfolio}, we add a volatility-dependent penalty term to the {\cblue DA }SCUC objective that reflects {\cblue the RT uncertainty and RT consumer exposure} during a period of the day that is expected to be particularly {\cblue volatile} (e.g., peak hours on a high-load day).
Thus, {\cblue the DA SCUC is completed with} an observed real-world phenomenon, namely {\cblue RT} price spikes.
\I We account for volatility by constructing data-driven adversarial scenarios based on stressing covariance factors. More specifically, we build a data-driven uncertainty set via PCA (principal-component analysis), relying on NYISO historical data. The uncertainty set captures locational {\cblue correlations} of uncertainty data, which {\cblue includes} both load and wind generation. {\cblue The leading modes of the covariance matrix are used to adversarially magnify volatility.  See Section \ref{sec: uncertainty_set}.}
\I We develop an algorithm to solve the proposed nonconvex, risk-aware {\cblue DA SCUC} model. The algorithm is similar to Benders' decomposition scheme \citep{bnnobrs1962partitioning};  it iteratively refines a relaxed {\cblue DA} problem by adding cuts generated from an adversarial {\cblue RT} problem. Due to the nonconvex subproblem, our algorithm cannot directly adopt Benders cuts as has been done in \cite{jiang2011robust} and \cite{bertsimas2012adaptive}. Instead, we use no-good cuts \cite{wolsey2020integer}, integer L-shaped cuts \citep{laporte1993integer}, and novel problem-specific logic-based Benders decomposition (LBBD) cuts \citep{hooker2003logic}. We also speed up the algorithm by developing a grid search heuristic for the nonconvex {\cblue RT} problem, and adding cuts via a branch-and-cut scheme. With the proposed solution approach, we are able to solve {\cblue a} large-scale NYISO case study efficiently.
\EI


\section{\cblue Risk-aware SCUC}\label{sec: setup}

Our modified SCUC computation builds on current industry practice \cite{nyiso_day_ahead} by extending common model formulations \cite{al2014role,mieth2022risk} to incorporate expected volatility in {\cblue RT} operations.
Under such conditions, a financial penalty may ensue and, possibly, actual physical risk may occur.  Empirical evidence from virtual trader activity suggests that on selected periods of the year (e.g., high temperature days) high-volatility time spans are predicted in advance \citep{eia_gov}. {\cblue We remind the reader that excess RT load (i.e., RT load in excess over its DA counterpart) is paid for at the RT LMP.  Formally, at each bus\footnotemark:
\begin{align} \label{eq:excess}
& \text{RT consumer exposure at bus $i$} \ = \ \nonumber \\ & \text{(RT LMP at $i$)} \times \max\{\text{\,(RT load at $i$ - DA load at $i$)},\, 0\,\},
\end{align}
where we are considering a particular time period of the day. 

\footnotetext{We consider payments settled on a nodal/bus basis. This could be extended to zonal/area payments, which are practiced in some power markets.}

At the same time, DA markets imply a different financial obligation, namely
\begin{align} \label{eq:dapayments}
& \text{DA payment at bus $i$} \, = \,  \nonumber \\ & \quad \text{(DA LMP at $i$)} \times \text{\,(DA load at $i$)}.
\end{align}
When RT LMPs are (much) higher than DA LMPs, and likewise with loads, the RT consumer exposure is very high -- an undesirable outcome.}

To handle both financial and physical risk, while maintaining its overall structure, we modify the \textit{objective} of the SCUC mixed-integer program (MIP) by adding a term that approximates {\cblue RT consumer exposure due to volatility -- in renewable output and loads -- during the time period of interest, plus appropriate constraints.}  

The algorithmic implementation of this updated SCUC can be viewed as a two-stage model. {\cblue In the first stage a unit commitment, DA dispatch and prices are obtained. In the second stage, the first-stage solution to an appropriately 
instrumented adversary. The goal of the adversary is to compute a realistic scenario for the critical time period that \textit{stresses} the given SCUC solution.} This stress computation relies on data-driven models of volatility in loads and wind power.  The information gleaned from this process is incorporated, in appropriate form, into {\cblue our revised SCUC computation.}

This cycle is repeated until a desirable convergence is attained.  Figure \ref{fig: two-stage_model} illustrates the relationship between the two stages of the model. The model in Figure \ref{fig: two-stage_model} has two critical features. First, it does not depart from the standard SCUC paradigm in that it produces precisely the same output as SCUC does, namely, a UC decision and DA dispatch. The UC decision can thus be used to obtain {\cblue DA} {\cblue locational} marginal prices \citep{ISO-NE_LMP_FAQ} via dual variables, as done in practice via the so-called ``pricing run". Second, our algorithm does not simply amount to a robust optimization approach in the sense that the objective function is a blend of {\cblue DA} costs and {\cblue RT} robustified costs, with the importance assigned to the second term controlled by a multiplier that reflects risk tolerance. 

 
\begin{figure}[hbpt]
\centering
\begin{adjustbox}{minipage=\linewidth,scale=0.9}
\begin{tikzpicture}[node distance = 5cm,auto]
\node [block, text width=2.5cm] (master) {Day-Ahead SCUC Problem};
\node [block, right of=master, text width=3.5cm] (sub) {Real-Time Adversarial DCOPF Problem};
\draw[{Latex[length=1.5mm]}-] (sub) to [bend right=30] node [above, sloped] (text1) {Unit commitment decisions} (master);
\draw[{Latex[length=1.5mm]}-] (master) to [bend right=30] node [below, sloped] (text1) {Adversarial scenario} (sub);       
\end{tikzpicture}
\end{adjustbox}
\caption{Two-stage risk-aware optimization model}\label{fig: two-stage_model}
\end{figure}
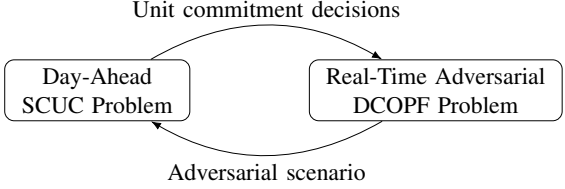
\subsection{Model for Risk-Aware SCUC Problem}
The proposed SCUC assumes the viewpoint of the system operator. In our model for the {\cblue DA} market, we minimize the total cost of electricity production and generator start-up and shutdown, plus an estimate of the {\cblue RT consumer exposure} attained in the {\cblue RT} market, which are due to incorrect wind and/or net-load forecasts during a time period of interest. We further stress that we specifically focus on errors in forecast that are due to volatility
{\cblue
(We note that forecast errors from behind-the-meter solar generation are contained in the net-load forecast errors. Through a straightforward extension, our model can also explicitly consider large-scale solar generation but we omit this option for ease of exposition).}
The {\cblue RT consumer exposure} is the payment incurred by the consumers in the {\cblue RT} market when load is underestimated in the {\cblue DA} market. Note that the {\cblue RT consumer exposure} is, at the point of time when SCUC is run, an uncertain quantity. 

In the formulation below, an \textit{appropriately robust} estimate of the {\cblue RT consumer exposure} is indicated by the quantity $\hat V$, which appears in the objective function \eqref{eq: uc_0} and in constraint \eqref{eq: fullproblem_2}. Note that in the objective function $\hat V$ is scaled by a certain weight $\rho > 0$ and the model becomes more risk-aware as $\rho$ increases. 
Hence, we model the total cost as follows:
\begin{align}\label{eq: uc_0}
&\sum_{t\in\Tda}\Big(\sum_{g\in\G} \left(h_{gt} + \cstart v_{gt} + \cdown w_{gt}\right)\nonumber\\&+ \sum_{i\in\N} \cvoll \punmet\Big)+\rho \hat V,
\end{align}
where the first four terms denote the production cost $h_{gt}$, startup cost $\cstart$, shutdown cost $\cdown$, and {\cblue DA} value of lost load (VOLL) $\cvoll$. In addition, $\Tda$ is the set of time periods (hours) in the {\cblue DA} market, $\G$ is the set of generators, $\N$ is the set of buses. Binary variables $v_{gt}$ and $w_{gt}$ denote startup and shutdown decisions, and $\punmet$ is the unsatisfied load. Note that $h_{gt}$ is a piecewise linear function of power output $p_{gt}$, and can be modelled with constraints in \eqref{eq: uc_1}. Each linear piece ($o\in\mc{O}$) in the piecewise linear function is represented by one constraint ($\forall t\in\Tda$):
\begin{align}\label{eq: uc_1}
h_{gt} \geq C^1_{og} p_{gt} + C^0_{og}y_{gt}&&\forall o\in\mc{O}, g\in\Gt,
\end{align}
where binary variable $y_{gt}$ denotes the commitment decision, $C^1_{og}$ and $C^0_{og}$ are respectively the slope and intercept for the cost segment $o$ in the piecewise linear function, and $\Gt$ is the set of thermal generators.

In addition, we have the following constraints in the {\cblue DA} problem:

(1) {\it Load Constraints ($\forall t\in\Tda)$: }For each hour the {\cblue DA} expected load $\dbar$ at each node is either satisfied by production at the node and power transmitted to the node, or by shedding load penalized by the VOLL. The set $\G_i$ includes both thermal and renewable (wind) generators at bus $i$, $\L$ is the set of transmission lines, and $f_{ijt}$ is the power flow.
\begin{subequations}\label{eq: uc_2}
\begin{alignat}{4}
& \sum_{g\in\G_i} p_{gt} + \punmet+ \sum_{(j,i)\in\L}f_{jit}- \sum_{(i,j)\in\L} f_{ijt} = \dbar\hspace{-1cm}\nonumber\\&&&\forall i\in\N\label{eq: uc_2_1}\\
& \punmet \geq 0&&\forall i\in\N. \label{eq: uc_2_2}
\end{alignat}
\end{subequations}

(2) {\it Linearized (DC) Power Flow Constraints ($\forall t\in \Tda$): }Constraint \eqref{eq: uc_3_1} defines the linearized (DC) power flow in terms of the voltage angle difference $\theta_{it}-\theta_{ji}$ at time $t$ between buses $i$ and $j$ and line susceptance $B_{ij}$. Constraints \eqref{eq: uc_3_2} and \eqref{eq: uc_3_3} bound flows.
\begin{subequations}\label{eq: uc_3}
\begin{alignat}{4}
& f_{ijt} = B_{ij} \left(\theta_{it} - \theta_{jt}\right) ~~&&\forall (i, j)\in\L \label{eq: uc_3_1} \\
& f_{ijt} \leq \ptrans&&\forall (i, j)\in\L \label{eq: uc_3_2}\\
& f_{ijt}\geq - \ptrans&&\forall (i, j)\in\L. \label{eq: uc_3_3}
\end{alignat}
\end{subequations}

(3) {\it Startup/Shutdown Decisions}: The following constraints link on/off statuses with startup and shutdown decisions.
\begin{subequations}\label{eq: uc_4}
\begin{alignat}{4}
& v_{gt} - w_{gt} = y_{gt} - y_{g, t-1} \quad && \forall g\in\Gt, t\in\Tda\sm\{1\}\\
& v_{gt}, w_{gt}, y_{gt} \in \{0,1\}\quad && \forall g\in\Gt, t\in\Tda.
\end{alignat}
\end{subequations}

(4) {\it Production Constraints ($\forall t\in\Tda$): }Constraints \eqref{eq: uc_5_1} and \eqref{eq: uc_5_2} set bounds for power outputs of thermal generator $g$. Constraints \eqref{eq: uc_5_3} set the outputs of wind generator $g$ at no more than $\pmaxbar$, where $\pcurtail$ is the curtailed output, and $\Gw$ is the set of wind generators.
\begin{subequations}\label{eq: uc_5}
\begin{alignat}{4}
&p_{gt} \leq P^\max_g y_{gt} \qquad &&\forall g\in\Gt \label{eq: uc_5_1}\\
& p_{gt} \geq P^\min_g y_{gt} \qquad && \forall g\in\Gt\label{eq: uc_5_2}\\
& p_{gt} + \pcurtail = \pmaxbar \qquad && \forall g\in\Gw\label{eq: uc_5_3}\\
& p_{gt} \geq 0 \qquad && \forall g\in\G\label{eq: uc_5_4}\\
& \pcurtail \geq 0 && \forall g\in\Gw.\label{eq: uc_5_5}
\end{alignat}
\end{subequations}

(5) {\it Ramping Constraints ($\forall t\in\Tda\sm\{1\}$): }The following constraints enforce generator multi-period ramping limit $M_g$.
\begin{subequations}\label{eq: uc_ramp}
    \begin{alignat}{4} \label{eq: fullproblem_1_3} 
    & p_{gt} - p_{g, t-1} \leq M_g y_{g, t-1} + P^\min_g v_{gt}\qquad &&\forall g\in\Gt\\
    & p_{g,t-1} - p_{gt} \leq M_g y_{gt} + P^\min_g w_{gt}\qquad &&\forall g\in\Gt
    \end{alignat} 
\end{subequations}

(6) {\it Risk-Aware {\cblue RT Consumer Exposure} Model: }Lastly, we address the {\cblue RT consumer exposure} in {\cblue RT}. { As described above we focus, in particular, on a critical and impactful set of time periods $\Trt$ which is \textit{expected} to have high load and wind generation volatility}. We assume that such time periods can be forecasted with significant certainty at the time of the SCUC computation -- as stated above, empirical evidence suggests that this capability already exists. Our model provides a robust estimate for the  {\cblue RT consumer exposure} during these critical time periods, taking into account the volatility of load and wind generation. 
{\cblue By focusing on these time periods, the model protects the consumers from extremely high excess {\cblue RT} payments while maintaining tractability.}
 
 More specifically, if the load at some bus $i$ and time $\tstar$ is higher in {\cblue RT} than anticipated in the first-stage, then the consumer needs to procure the difference at its {\cblue RT} LMP. However, there is no {\cblue RT} penalty for overestimating load in the first stage. 
More precisely, suppose that $\vomega$ indicates a \textit{realization} of all quantities that are uncertain at the {\cblue DA} phase, and known in {\cblue RT}, including in particular the vector $\drt$ of {\cblue RT} loads. Let $\lambdayo$ be the corresponding {\cblue RT} LMP, where the notation stresses the dependence on the unit-commitment vector $\vy$ and the uncertain quantities $\vomega$. Then the {\cblue total RT consumer exposure at all buses,} averaged over the number of time periods per hour in the critical time period $\Trt$, given realization $\vomega$, is given by the following expression: 
\begin{align}\label{eq: uc_6}
    V = V(\vy, \vomega) = \frac{1}{\Nint} \sum_{t\in\Trt}\sum_{i\in\N} \lambdayo(\drt - \dstarbar)^+.
\end{align}
Note that unlike $\Tda$, the time intervals in $\Trt$ could have a finer granularity than 1 hour, as in practice {\cblue RT} models may be solved every 5 or 15 minutes. Let $\Trth$ be the set of hours in $\Trt$. Then $\Nint = \frac{|\Trt|}{|\Trth|}$ is the number of time periods per hour. For example, if the {\cblue RT} time interval is 5 minutes, then there are 12 time intervals per hour. 
We point out that \eqref{eq: uc_6} is nonconvex because of the bilinear term $\lambdayo \drt$ and the expression $(\drt - \dbar)^+$.
{\cblue We also note that \eqref{eq: uc_6} considers the total RT consumer exposure. Additional regularization terms or fairness considerations \cite{xinying2023guide} are possible but beyond the scope of this paper.}

Several quantities in this expression are unknown at the time of the {\cblue DA} computation, i.e., for a given $\vy$; therefore, $V$ is a random variable. Thus, a major component of our algorithm is the estimation of $\hat V$, the risk term in \eqref{eq: uc_0}, as a \textit{robust} estimate for $V$, through appropriate constraints to be added, incrementally, to the formulation of our risk-averse SCUC problem. Section \ref{subsec: advrealtime} describes the robust modeling of $V$ and algorithmic details are deferred to   Section \ref{sec: solution}.

(7) {\it Full Formulation:}  We now present the risk-aware SCUC optimization problem in full.
\begin{subequations}\label{eq: fullproblem}
\begin{alignat}{4}
\min~&\sum_{t\in\Tda}\Big(\sum_{g\in\G} \left(h_{gt} + \cstart v_{gt} + \cdown w_{gt}\right)\nonumber\hspace{-3cm}\\&+ \sum_{i\in\N} \cvoll \punmet\Big)+\rho \hat V (\vy)\hspace{-4cm}\\
\st~&\eqref{eq: uc_1},\eqref{eq: uc_2},\eqref{eq: uc_3},\eqref{eq: uc_5}  &&\forall t\in\Tda\label{eq: fullproblem_1_1}\\
& \eqref{eq: uc_4} \label{eq: fullproblem_1_2}\\
& \eqref{eq: uc_ramp}&&\forall t\in\Tda\sm\{1\}\label{eq: fullproblem_1_3}\\
& \hat V(\vy) \, = \, \max_{\vomega\in\Omega} \frac{1}{\Nint}\sum_{t\in\Trt}\sum_{i\in\N} \lambdayo (\drt - \dstarbar)^+ \hspace{-3.5cm}\nonumber\\&&\label{eq: fullproblem_2}
\end{alignat}
\end{subequations}
where $\Omega$ is an uncertainty set that includes \textit{reasonable, but stress-revealing} {\cblue RT} parameters -- in Section \ref{sec: uncertainty_set} we will indicate how this set $\Omega$ is constructed in a data-driven manner. Constraints \eqref{eq: fullproblem_1_1} - \eqref{eq: fullproblem_1_3} include all constraints in a standard SCUC problem. Constraint \eqref{eq: fullproblem_2} provides a robust estimate for the {\cblue RT consumer exposure}. Note that this constraint models the maximum of the {\cblue RT consumer exposure} expression \eqref{eq: uc_6} over the uncertainty set{\cblue , reflecting the worst-case cost of volatility}. To handle this nonconvex constraint, we describe a decomposition algorithm that replaces \eqref{eq: fullproblem_2} with linear cutting planes in Section \ref{sec: decomposition}. 
\subsection{Adversarial Real-time Operations} \label{subsec: advrealtime}
We now describe our methodology for attaining a measure of robustness in our estimate $\hat V = \hat V(\vy)$, which is the (uncertain) {\cblue RT consumer exposure} $V$ given a UC decision $\vy$. To that end, we next describe our formulation for the uncertain {\cblue RT} DC optimal power flow (DCOPF) problem. We consider uncertainty in both {\cblue RT} load $\drt$ and {\cblue RT} wind generation $\prt$, and let $\vdrt$ and $\vprt$ be their corresponding vectors.  We assume that the pair $(\vdrt, \vprt)$ belongs to the uncertainty set $\Omega$, and thus $\vomega = (\vdrt, \vprt)$.
Given a set of commitment decisions, we are interested in the pair $(\vdrt, \vprt) \in \Omega$ that attains the peak {\cblue RT consumer exposure}. To  construct such a pair, 
we denote the vector of the DCOPF problem decision variables as $\vxrt$, the vector of all fixed commitment decisions (selected {\cblue DA}) as $\vystar$, the DCOPF feasible region as $\Xrt(\vystar, \vomega)$, which is parameterized by commitment decisions and uncertain quantities, and the DCOPF problem itself as $\dcopf$, which is defined as follows:
\begin{align}\label{eq: dcopf_0}
    \min_{\vxrt\in\Xrt(\vystar, \vomega)}\sum_{t\in\Trt}\left(\sum_{g\in\mc{G}} h_{gt}+ \sum_{i\in\mc{N}} \cvollrt \punmet\right).
\end{align}
Note that we set the {\cblue RT} VOLL $\cvollrt > \cvoll$ to penalize unsatisfied {\cblue RT} load.

Next, the right-hand side (RHS) of the load constraint \eqref{eq: uc_2_1} is replaced with {\cblue RT} load $\drt$:
\begin{align}\label{eq: dcopf_1}
    &\sum_{g\in\G_i} p_{gt} + \punmet+ \sum_{(j,i)\in\L}f_{jit}- \sum_{(i,j)\in\L} f_{ijt} = \drt\hspace{-1.5cm}\nonumber\\&&&\forall i\in\mc{N}.
\end{align}
The dual of this constraint is the {\cblue RT} LMP $\lambda_{it}$.

In addition, we fix the commitment decisions in constraints \eqref{eq: uc_1}, \eqref{eq: uc_5_1} and \eqref{eq: uc_5_2}, and replace the RHS of \eqref{eq: uc_5_3} with {\cblue RT} wind generation:
\begin{subequations}\label{eq: dcopf_2}
\begin{alignat}{4}
& h_{gt} \geq C^1_{og} p_{gt} + C^0_{og}y^*_{gt}&&\forall o\in\mc{O}, g\in\Gt \label{eq: dcopf_2_0}\\
&p_{gt} \leq P^\max_g \ystar \qquad &&\forall g\in\Gt \label{eq: dcopf_2_1}\\
& p_{gt} \geq P^\min_g \ystar \qquad && \forall g\in\Gt\label{eq: dcopf_2_2}\\
& p_{gt} + \pcurtail = \prt \qquad && \forall g\in\Gw.\label{eq: dcopf_2_3}
\end{alignat}
\end{subequations}

To reformulate \eqref{eq: uc_ramp}, note that in real time the ramping rate $M_g$ is prorated to $\Mrt$ based on the time intervals in $\Trt$. Denoting the {\cblue RT} limit on ramping as $\hat{M}_g$. Then $\hat{M}_g = \Mrt$ if $y^*_{g,t-1}=y^*_{g,t} = 1$, i.e., when the generator is on for both the previous and current time intervals; Otherwise, $\hat{M}_g=P^\min_g$. Constraints \eqref{eq: uc_ramp} are reformulated as:
\begin{subequations}\label{eq: dcopf_ramp}
\begin{alignat}{4}
    & p_{gt} - p_{g, t-1} \leq \hat{M}_g \qquad &&\forall g\in\Gt\\
    & p_{g,t-1} - p_{gt} \leq \hat{M}_g\qquad &&\forall g\in\Gt
\end{alignat}
\end{subequations}

Hence, the feasible region of the DCOPF problem, given the UC decision $\vy^*$ and realization $\vomega$, is defined as follows:
\[
  \Xrt(\vy^*, \vomega) := \Set{ \vxrt | \begin{aligned}
    & \eqref{eq: uc_2_2}, \eqref{eq: uc_3},\eqref{eq: uc_5_4}, \eqref{eq: uc_5_5}\\
    &\eqref{eq: dcopf_1}, \eqref{eq: dcopf_2},\eqref{eq: dcopf_ramp}\end{aligned}~~,\forall t\in\Trt}.
\]

Above, to simplify notation, we re-use notation for {\cblue DA} decision variables (e.g. $h_{gt}$ and $p_{gt}$) in the {\cblue RT} problem, in which case they represent {\cblue RT} decisions.  In what follows, the dual of the DCOPF problem given $\vystar$ and $\vomega$
will be denoted by $\dcopfd$.

We now describe how to implement constraint 
\eqref{eq: fullproblem_2}.  First, the {\cblue RT} LMP $\lambda_{it}$ (again given  $\vy^*$ and $\vomega$) is an optimal solution to $\dcopfd$. Thus, \eqref{eq: fullproblem_2} amounts to finding uncertain quantities and resulting {\cblue RT} LMPs that maximize the {\cblue RT consumer exposure} in the critical time period $\Trt$, as given in \eqref{eq: uc_6}. We define the adversarial DCOPF problem as:
\begin{subequations}\label{eq: dcopf_adv}
\begin{alignat}{4}
\dcopfADV :=\hspace{-2.5cm}\nonumber\\
    &\max_{\vomega \in \Omega} ~&&\frac{1}{\Nint}\sum_{t\in\Trt}\sum_{i\in\N}\lambdart (\drt - \dbar)^+\label{eq: dcopf_adv_obj}\\
    &\st~&&\lambdart \text{~is an optimal solution of~}\dcopfd.\label{eq:dcopf_adv_1}
\end{alignat}
\end{subequations}

One way to write constraint \eqref{eq:dcopf_adv_1} explicitly, is to write a primal-dual formulation for $\dcopf$ and utilize strong duality, as in \citep{dvorkin2017ensuring}. However, 
since \eqref{eq: dcopf_adv_obj} and the resulting strong duality constraint are nonlinear, $\dcopfADV$ may be difficult to solve. As an alternative, in Section \ref{sec: maxmin} we present a grid search method to approximately solve $\dcopfa$ and find adversarial stressors.

\subsection{Data-Driven Uncertainty Set}\label{sec: uncertainty_set}
Our choice of the uncertainty set $\Omega$ is driven by the the goal to adapt the SCUC computation to better deal with \textit{volatility}. Toward this end, our uncertainty set $\Omega$ is constructed through an approximate, PCA (principal component analysis)-driven representation of the \textit{covariance} of loads and renewable outputs. Both load and renewable exhibit locational correlation and we find that a few leading modes of the covariance matrix are sufficient to explain total variability \citep{liang2022weather}. The particular methodology we follow is motivated by the ``factor stressing" technique employed in the financial services industry \cite{haughfactor}.


More specifically, for each time $t\in\Trt$ we obtain data for {\it recent past} (e.g., hourly observations {\cblue of one month of data leading up to the modeled day}), and use such data to construct a covariance matrix. The number of rows of the load covariance matrix equals $|\mc{N}|$ and the number of rows of the wind covariance matrix equals the number of wind farms. We conduct a spectral decomposition on a covariance matrix to obtain its eigenvalues and standardized eigenvectors. The eigenvectors corresponding to the $K$ largest eigenvalues (sorted from large to small), i.e., the $K$ leading modes, are used to construct our set $\Omega$. 

Let us consider loads (a similar process is used for wind). Let $k \in\{1,\hdots,K\}$ be the index for the covariance matrix modes (in leading order, i.e., the first mode corresponds to the largest eigenvalue), and $\vqd := (Q^{d}_{k1t}, Q^{d}_{k2t}, \hdots, Q^d_{k|\mc{N}|t})^\top$ be the $k^{th}$ eigenvector for the covariance matrix at time $t$. For each $k \le K$ we will adversarially compute a quantity $\alphad \ge 0$, which is the ``stressor" that magnifies the load variability along the $k^{th}$ leading mode at time $t$. Similarly, $\vqw$ and $\alphaw$ are the eigenvector and stressor of wind power. 

Using this notation we can describe the formal set $\Omega$ from which we select our load $\drt$ and wind power $\prt$ scenarios.  It is given by the pairs $(\drt, \prt)$ satisfying:
\begin{subequations}\label{eq: sub_adversarial}
\begin{alignat}{4}
& \drt = \dbar + \sum_{k=1}^K Q^d_{kit} \alpha^d_{kt} \hspace{-1.8cm}&&&\forall i\in\mc{N}, t\in\Trt\label{eq: sub_adversarial_0}\\
& \prt = \pmaxbar + \sum_{k=1}^K Q^w_{kgt} \alpha^w_{kt} \hspace{-4.5cm}\nonumber\\& &&&\hspace{-5.5cm}\forall g\in\Gw, t\in\Trt \label{eq: sub_adversarial_1}\\
& |\sum_{k=1}^K \alpha^{ind}_{kt}| \leq \Sigma^{ind} &&&\forall t\in\Trt, ind\in\{d, w\}\label{eq: sub_adversarial_2}\\
& |\alpha^{ind}_{kt}|\leq R^{ind}  &&&\hspace{-0.4cm} \forall k=1,\hdots,K; t\in\Trt, ind\in\{d, w\}\label{eq: sub_adversarial_3}\\
& \drt \geq 0 &&&\forall i\in\mc{N}, t\in\Trt\label{eq: sub_adversarial_4}\\
& \prt \geq 0 &&& \forall g\in\Gw, t\in\Trt\label{eq: sub_adversarial_5}.
\end{alignat}
\end{subequations}
Constraints \eqref{eq: sub_adversarial_0} represent the {\cblue RT} load scenarios as the sum of the forecast and stressed leading modes. Similar expressions are derived for wind power in constraints \eqref{eq: sub_adversarial_1}. Constraints \eqref{eq: sub_adversarial_2} and \eqref{eq: sub_adversarial_3} bound the stressors and control the level of conservatism{\ccblue , where $R^{ind}$ is the bound for stressors. }{\cblue Note that $\Sigma^{ind}$ denotes the bound for the absolute value of summation of stressors at each time period. In practice, it can be determined by historical observations.} Constraints \eqref{eq: sub_adversarial_4} and \eqref{eq: sub_adversarial_5} ensure load and wind power are non-negative.

{\cblue In sum, to construct a vector of load scenarios $\vdrt_t$ at $t\in\mc{T}$, we first compute a covariance matrix based on historical load data. We obtain K eigenvectors ($\vqd$, k = 1,...,K) corresponding to the K largest eigenvalues of this covariance matrix as the leading modes. The vector $\vdrt_t$ is the sum of load forecast and stressed leading modes, i.e., $\vdrt_t = \bf{\bar{D}}_t + \sum_{k=1}^K$ 
$\alphad$ $\vqd$, where the stressors are variables controlled by \eqref{eq: sub_adversarial_2} and \eqref{eq: sub_adversarial_3}. The wind production scenarios are similarly constructed. The uncertainty set is thus given by $\Omega = \{(\vdrt, \vprt)|\eqref{eq: sub_adversarial}\}$.
We note that our method can be modified to use a single covariance matrix to generate scenarios for the entire vector of uncertain parameters (i.e., load and wind).
We opt to model load and wind independently, i.e., using two separate covariance matrices, to reflect the typically insignificant correlation between load and wind forecast errors and to better study their individual impact on the solution.

Our uncertainty set captures the locational correlation without substantially increasing the number of variables. In comparison, the polyhedral uncertainty set in \cite{bertsimas2012adaptive} independently generates scenarios in each location, thus ignoring locational correlation. The data-driven uncertainty set proposed in \cite{velloso2019two} captures both spatial and temporal dependence of the renewable production by expressing it as a convex combination of historical observations. To effectively represent extremely adverse scenarios, this approach potentially needs to include a large number of past observations, each requiring a new variable. The ellipsoidal uncertainty set also captures covariance information. It consists of second-order cone constraints and can be defined via the lower triangular matrix of the Cholesky decomposition for the inverse-covariance matrix \cite{johnstone2021conformal}, while our uncertainty set has a simpler and more intuitive structure. In \cite{zhao2022sustainable}, scenarios are generated from a data-driven uncertainty set based on PCA, where stressors are modeled with {\ccblue kernel} density estimation (KDE). In contrast, we impose constant bounds on the stressors, reducing the number of variables and allowing a faster computation and the use of a grid search heuristic to speed up the algorithm, as described in Section~\ref{sec: maxmin}.} 
\section{Solution Approach}\label{sec: solution}
Solving the risk-aware optimization model \eqref{eq: fullproblem} directly (e.g., using off-the-shelf solvers) is difficult because the first-stage problem contains both integer variables and bilinear terms (in constraint \eqref{eq: fullproblem_2}). 
In Section \ref{sec: decomposition}, we show how to solve this two-stage problem \eqref{eq: fullproblem} by relying on an iterative decomposition algorithm that approximates constraint \eqref{eq: fullproblem_2} with increasing accuracy. The procedure is guaranteed to finitely terminate with an optimal solution to \eqref{eq: fullproblem}. Section \ref{sec: cuts} introduces the cutting planes used in the procedure. Section \ref{sec: maxmin} describes an adversarial procedure that is used to improve the accuracy of the approximate formulation. We include
 additional implementation details in Section \ref{sec: impl_details}.
\subsection{Decomposition Algorithm}\label{sec: decomposition}
To solve \eqref{eq: fullproblem}, we rely on an iterative method which is similar to the Benders' decomposition scheme \cite{bnnobrs1962partitioning}. We will next provide a simplified version of this method, and later we will amend this outline in order to handle integer variables.

Under this simplified procedure, at any intermediate point of the algorithm we will have a relaxation for \eqref{eq: fullproblem} that we term the ``master problem" which is initially obtained from \eqref{eq: fullproblem} by removing \eqref{eq: fullproblem_2} and replacing it with $\hat V \ge 0$. As the procedure iterates, cuts are added to the master problem to approximate  \eqref{eq: fullproblem_2} with increasing accuracy.

Suppose that at some iteration the solution to the master problem is given by vector $\vy^*$. This vector may not solve problem \eqref{eq: fullproblem} since it may not satisfy \eqref{eq: fullproblem_2}.  To check whether this is the case, $\vy^*$ is input to a procedure given in Section \ref{sec: maxmin}. If this procedure verifies that $\vy^*$ is feasible for \eqref{eq: fullproblem_2}, then $\vy^*$ is optimal for \eqref{eq: fullproblem}. Otherwise, a cutting plane is generated for strengthening the master problem relaxation.
Figure \ref{fig: decomposition} provides an illustration of the procedure, which is a decomposition algorithm. In this figure, $\Pmaster$ is the master problem and $\dcopfa$ is the adversarial problem given $\vy^*$. Additional algorithmic details such as grid search, partially fixing $\vy$, and branch-and-cut will be introduced in later sections.


More specifically, the initial master problem $\Pmaster$ is:
\begin{align*}
\Pmaster := \min~~&\eqref{eq: uc_0}\\
\st~~&\eqref{eq: uc_1},\eqref{eq: uc_2},\eqref{eq: uc_3},\eqref{eq: uc_5}  &&\forall t\in\Tda\\
& \eqref{eq: uc_4}, \Vhat\geq 0\\
& \eqref{eq: uc_ramp}&&\forall t\in\Tda\sm\{1\}.
\end{align*}
   This master problem includes the integer unit-commitment variables, which necessitates an appropriate adjustment to the proposed methodology, as explained in Section \ref{sec: impl_details}.
 
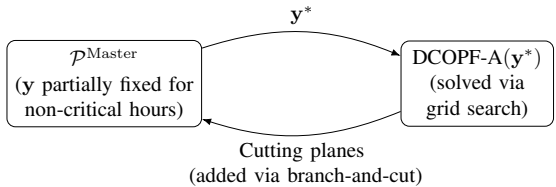
\begin{figure}[hbpt]
\centering
\begin{adjustbox}{minipage=\linewidth,scale=0.8}
\begin{tikzpicture}[node distance = 5cm,auto]
\node [block, text width=3cm] (master) {$\Pmaster$\\ ($\vy$ partially fixed for non-critical hours)};
\node [block, right=3.3cm of master, text width=2.3cm] (sub) {$\dcopfa$\\(solved via grid search)};
\draw[{Latex[length=1.5mm]}-] (sub) to [bend right=20] node [above, midway, align=center] (text1) {$\vy^*$} (master);
\draw[{Latex[length=1.5mm]}-] (master) to [bend right=20] node [midway,below,align=center] (text1) {Cutting planes\\~(added via branch-and-cut)} (sub);       
\end{tikzpicture}
\end{adjustbox}
\caption{The decomposition algorithm}\label{fig: decomposition}
\end{figure}
\subsection{Cutting Planes}\label{sec: cuts}
The cutting planes we use include no-good  \cite{jiang1994nogood}, integer L-shaped  \cite{laporte1993integer}, and LBBD cuts \cite{hooker2003logic}. Together these cuts yield a lower bound for the worst-case {\cblue RT consumer exposure} $V$ given UC decision $\vy$. 
These cuts have the generic form $\Vhat(\vy) \geq f(\vy)$, where $f(\vy)$ is a linear function of the commitment decisions $\vy$, and satisfies the following conditions: 
\begin{itemize}
\item For all feasible UC decisions of $\vy$, $f(\vy)$ must be an underestimator for the worst-case {\cblue RT consumer exposure} given $\vy$, i.e., $\max_{\vomega\in\Omega} \frac{1}{\Nint}\sum_{t\in\Trt}\sum_{i\in\N} \lambdayo(\drt - \dstarbar)^+ $.  Moreover,
the inequality $\hat V(\vy) \ge f(\vy)$ should cut off the current master problem solution.
\item Let $\vy^*$ be the commitment decisions in the current master problem solution. In a neighborhood of $\vy^*$, $f(\vy)$ should, in addition, be a close lower bound to the worst-case {\cblue RT consumer exposure} given $\vy$.  The strength of the cut rests on how close this approximation is.
\end{itemize}

We first present a no-good cut which only provides a good lower bound at the current solution $\vy^*$, then we develop integer L-shaped cuts and LBBD cuts that aim to improve this lower bound for solutions in a neighborhood of $\vy^*$. Throughout, we assume that the {\cblue worst-case {\cblue RT consumer exposure} $\uhat$ given $\vystar$ is attained by profile $\bm{\omega}^*$ (including loads $\drthat$ as per DCOPF-A($\bm{y}^*$) in \eqref{eq:dcopf_adv_1}), i.e., that $\uhat = \uhat(\vystar) :=   
\frac{1}{\Nint}\sum_{i\in\N}\sum_{t\in\Trt}\lambdaystaromegastar(\drthat - \dbar)^+$.}

\subsubsection{No-good Cuts}
Let $\vy^*$ be the commitment decisions in the current master problem solution. Define $\Ione = \Set{g|y_{gt}^* = 1, \forall g\in \Gt}$ and $\Izero = \Set{g|y_{gt}^* = 0, \forall g\in \Gt}$ respectively as the set of thermal generators that are on/off at time $t$. Let $\drthat$ and $\lambdahat$ be the optimal primal and dual solution values, in $\dcopfa$, for $\drt$ and $\lambda_{it}$, respectively. The no-good cut is as follows:
\begin{align}\label{eq: no_good}
\Vhat(\vy) \geq &\nonumber \\ &\hspace{-1cm}\uhat(\vystar)\left(1- \sum_{t\in\Trth}\left(\sum_{g\in \Ione} (1-y_{gt}) + \sum_{g\in \Izero} y_{gt}\right)\right).
\end{align}
When $y_{gt} = y^*_{gt}, \forall g\in\Gt, t\in\Trth$, the cut can be simplified to $\Vhat(\vy) \geq \uhat (\vystar)$, which provides the correct value for $\Vhat(\vy)$. Otherwise, when at least one generator in $\Ione$ is turned off or at least one generator in $\Izero$ is turned on during $\Trt$, the RHS of the cut becomes non-positive.

The no-good cut is not very strong as it only provides a good underestimator at the current master solution. In fact, as there are $2^{|\Gt||\Trth|}$ possible values for the vector $\vy$, the algorithm could run through $\O(2^{|\Gt||\Trth|})$ iterations to find the optimal solution if only no-good cuts are used.
\subsubsection{Integer L-Shaped Cuts}
The integer L-shaped cut strengthens the no-good cut by improving a lower bound at some feasible commitment solutions in a neighborhood of $\vy^*$:
\begin{align}\label{eq: l_shaped}
\Vhat(\vy) \geq &\nonumber\\& \hspace{-1cm}\uhat(\vystar) + a\sum_{t\in\Trth}\left(\sum_{g\in\Ione}y_{gt} - \sum_{g\in\Izero}y_{gt} - |\Ione|\right),
\end{align}
where $a = \max(\uhat - \hat V_1, (\uhat - \hat V_0)/2)$, $\hat V_1$ equals the minimum value of the {\cblue RT consumer exposure} when exactly one generator changes its commitment decision (i.e., when $\sum_{g\in\Ione}y_{gt} - \sum_{g\in\Izero}y_{gt} = |\Ione| - 1$, which we call the {\it 1-neighbors} of $\vy^*$), and $\hat V_0$ is a lower bound on the {\cblue RT consumer exposure} under a feasible commitment decision (e.g.,  trivially $\hat V_0 = 0$ can be used).  

To see that the cut is valid, notice that for the 1-neighbors of $\vy^*$, the integer L-shaped cut provides a lower bound for $\Vhat(\vy)$, namely $\Vhat(\vy) \geq \min(\hat V_1, \frac{\uhat + \hat V_0}{2}).$ This lower bound is valid because $\min(\hat V_1, \frac{\uhat + \hat V_0}{2})\leq \hat V_1$.  And if there is more than one generator that changes its commitment decision, then the RHS of \eqref{eq: l_shaped} is no more than $\hat V_0$, and thus is valid. 

Note that to obtain $\hat V_1$, we could solve $\dcopfa$ for each $\vy$ among the 1-neighbors of $\vystar$, which requires solving $|\Trth|(|\Gt| - 1)$ subproblems. In our experiments, we observe that the integer L-shaped cuts indeed lead to faster convergence than the no-good cuts.  We point out that since $\min(\hat V_1, \frac{\uhat + \hat V_0}{2})\geq 0$, it is generally a better lower bound than that provided by a no-good cut.  However, the computational efforts to generate an L-shaped cut increases as more subproblems are solved to get $\hat V_1$.
\subsubsection{LBBD Cuts}
Another way to strengthen the no-good cut is to use the following LBBD cut:
\begin{align}\label{eq: lbbd-detailed}
& \Vhat(\vy) \geq \sum_{t\in\Trth}\left(\uhat_t(\vystar)\left(1-\sum_{g\in\Izero}y_{gt}\right)\right),
\end{align}
\noindent where $\uhat_t(\vystar) :=   
\frac{1}{\Nint}\sum_{\tau\in\Trt(t)}\sum_{i\in\N}\lambda_{i\tau}(\vystar,\vomega^*)(d^{\rm{RT}*}_{i\tau}-\bar{D}_{i\tau})^+$, with $\Trt(t)$ being the set of {\cblue RT} time periods in hour $t$. Note that $\uhat(\vystar) = \sum_{t\in\Trth}\uhat_t(\vystar)$. When any generator in $\Izero$ is turned on,  the term corresponding to hour $t$ in the RHS becomes non-positive. On the other hand, if all generators in $\Izero$ stay off, then this cut enforces that the {\cblue RT consumer exposure} is at least $\uhat_t(\vystar)$. Intuitively, if all generators in $\Izero$ and some generators in $\Ione$ are turned off, the {\cblue RT consumer exposure} is not likely to drop. In Proposition \ref{th: lbbd}, we prove that this lower bound is valid when the network is not congested and the ramping constraints are not binding in the {\cblue RT} market. As we will also explain, the cut is still good for our purpose even when those restrictions are relaxed.




\begin{proposition}\label{th: lbbd}
When there is no congestion in the network and the ramping constraints are not binding in the {\cblue RT} market, the LBBD cut \eqref{eq: lbbd-detailed} provides a correct lower bound for $\Vhat(\vy)$. Also, it provides the exact value of $\Vhat(\vy)$ at the current solution $\vy^*$.
\end{proposition}
\begin{proof}
Since there is no congestion, the LMPs at all nodes are equal, which we denote as $\lambda_t$. 

Consider the fixed wind and load profile $\vomega^*$ that defines the worst-case for $\vystar$ and fixed hour $t \in \Trth$. Let $\vy'$ be a feasible dispatch solution from the master problem, and $\Ionep$ and $\Izerop$ respectively be the corresponding set of generators that are on and off at $t$.  We will show that the {\cblue RT consumer exposure} term arising from hour $t$, namely $\uhat_t(\vystar)\left(1-\sum_{g\in\Izero}y_{gt}\right)$, is a lower bound for the {\cblue RT consumer exposure} at $t$, if we switch from $\vystar$ to $\vy'$. 
 We consider the following two cases:

{\it Case 1: } If $\Izerop \subset \Izero$ for some $t\in\Trth$, then $\sum_{g\in\Izero} y'_{gt}\geq 1$, which indicates that the RHS of \eqref{eq: lbbd-detailed} corresponding to $t$ is no more than 0. 

{\it Case 2: } If all generators that were off during $t$ under $\vy^*$ remain off under $\vy'$, then $\Izero \subseteq \Izerop$ and the RHS of \eqref{eq: lbbd-detailed} corresponding to $t$ remains unchanged. Thus, we need to show that this RHS value is a valid lower bound.

This is the same as showing that the LMP at every time $\tau\in\Trt(t)$ does \textit{not} decrease when we switch from $\vystar$ to $\vy'$.  To see this, note that any generator $g$ that is on at time $\tau$ under $\vystar$ must be of one of three types:
\begin{itemize}
\item [(a)] It defines the LMP, i.e., $h_{g,\tau} = \lambda_\tau$, or
\item [(b)] It satisfies $h_{g,\tau} < \lambda_\tau$, in which case the generator is operating at its maximum output, or
\item [(c)] The generator satisfies $h_{g,\tau} > \lambda_\tau$, in which case the generator is operating at its \textit{minimum} output level.
\end{itemize}
Moreover, the total load is equal to the maximum output of the generators in type (b), plus the minimum output from generators of type (c), plus the output of the generators of type (a). A similar characterization can be obtained for $\vy'$.  From this characterization it is
clear that if we turn off a generator the LMP cannot decrease, since turning a generator off reduces total available capacity, and thus the new LMP will be defined either by a generator of type (a) under $\vystar$ (in which case the LMP does not change) or by a generator of type (c) under $\vystar$, in which case the LMP strictly increases. The new LMP could also equal $\cvoll$, which is the highest value it can reach.

Finally, if $\vy'=\vy^*$, then  the RHS of \eqref{eq: lbbd-detailed} provides the exact value for the worst-case {\cblue RT consumer exposure}.
\end{proof}

Note that when there is congestion in the network, the result in Case 2 of the proof may not hold because prices could drop at certain locations even if only a strict subset of generators are turned on. {\cblue Consequently, the LBBD cut may overestimate the {\cblue RT consumer exposure}.} Also, if some of the ramping constraints are binding, then the production variables at $t-1$ and $t$ are coupled, which complicates the proof of Case 2, as the selection of marginal generator could be affected by the production levels of previous time periods (an example for this is provided in Appendix \ref{ch: ramp_example}).

Although we are not able to provide a formal proof for the quality of the cut in the general case, in our experiments we observe that using LBBD cuts with the decomposition algorithm returns correct solutions for the majority of instances. Note that the NYISO system in our case study does not have a lot of congestion. Also, $\Trt$ in Section \ref{ch: results_in_sample} of the case study contains a single time period and thus the {\cblue RT} problem does not have ramping constraints. Even if the cut leads to an overestimation (which is usually very small) of the {\cblue RT consumer exposure}, it may only lead to less adversarial stressors, which may still provide robust enough SCUC solutions.

The LBBD cut implies the no-good cut \eqref{eq: no_good}, and thus is strictly stronger than the no-good cut. Also, the LBBD cuts lead to much faster convergence compared with no-good cuts and do not require extra computational efforts to generate like for the integer L-shaped cuts. 
\subsection{Solving $\dcopfa$}\label{sec: maxmin}
Instead of directly solving the nonconvex problem $\dcopfa$, we employ grid search to solve it approximately. In grid search, we iterate through a set of fixed stressors (i.e, ``grids") $(\alphadtil, \alphawtil)$ that satisfy constraints \eqref{eq: sub_adversarial}, and find the stressors that lead to the highest {\cblue RT consumer exposure}. 

More specifically, for each fixed pair of stressors $(\alphadtil, \alphawtil)$, we obtain the corresponding scenario $\vomegatil$ using \eqref{eq: sub_adversarial_0} and \eqref{eq: sub_adversarial_1}, and solve $\dcopftil$. We then calculate the corresponding {\cblue RT consumer exposure}. Among all {\cblue RT consumer exposure} calculated in this way, we select the highest one as an estimate to the optimal value of $\dcopfa$, and the corresponding $\dcopftil$ problem provides the optimal solutions for adverse load and wind generation scenarios, and the LMPs.

The grid search method provides an approximation to the true objective of $\dcopfa$. It is difficult to directly solve the highly nonlinear and nonconvex $\dcopfa$ problem, even with state-of-the-art nonlinear optimization solvers such as Gurobi and Knitro. We also observe that the commonly-used McCormick relaxation is not very tight for nonlinear terms in $\dcopfa$. On the other hand, with a carefully-selected set of grids in the grid search, we can find adverse scenarios that lead to cost-saving UC decisions. We explain how to select the set of grids in Section \ref{sec: numerical_data}. {\cblue Note that thanks to the structure of the PCA-based uncertainty set, where uncertain values are correlated, we only need to generate grids for the stressors, and not for the original uncertain values. This greatly reduces the potential dimension of grids.}
\subsection{Implementation Details}\label{sec: impl_details}

The decomposition algorithm in Section \ref{sec: decomposition} requires solving a MIP master problem in each iteration, which is time consuming. Therefore, we instead add the cuts via branch-and-cut \cite{wolsey2020integer}, where we solve the master problem once and the violated cuts are added in the integral nodes of the branch-and-bound tree. {\cblue More specifically, when using an MIP solver to solve the master problem via branch-and-bound, one could check whether a node in the branch-and-bound tree has a feasible (thus integral) solution for the master problem. If so, then we can solve $\dcopfa$ with $\vystar$ fixed at this integral solution, and compare its optimal objective value with the estimated {\cblue RT consumer exposure} from the master problem solution at the current node. If the estimated value is not correct, then we add cuts to this node via the lazy callback function of the MIP solver. We repeat this process, until the stopping criteria for branch-and-bound are met.} We observe that branch-and-cut greatly improves the performance of our algorithm.

Among the three types of cuts derived in Section \ref{sec: cuts}, we can use any one of them or a combination of them in our algorithm. Note that integer L-shaped cuts and LBBD cuts are both stronger than no-good cuts, so there is no need to use no-good cuts if either of the two other cuts is used. Between integer L-shaped cuts and LBBD cuts, one does not dominate the other in terms of strength, so it could be helpful to include both cuts in the algorithm. Since it is time consuming to generate integer L-shaped cuts for large-scale problems, we use only LBBD cuts in our case study.

Compared with commitment decisions from the deterministic SCUC problem, the risk-aware solution usually keeps more capacities committed around critical hours that are represented by the {\cblue RT} market problem. Therefore, we can speed up the algorithm by only allowing commitment decisions around critical hours to deviate from its deterministic counterpart, which greatly reduces the search space of binary on/off decisions. {\cblue This is a heuristic and may lead to sub-optimal solutions. Nonetheless, we observe that our method performs well with this heuristic. Also, it is flexible and in practice we can allow more hours to deviate from deterministic solutions if needed.}

\section{NYISO Case Study}\label{sec: numerical}
\subsection{Data Resource and Simulation Environment}\label{sec: numerical_data}
For the numerical experiments we use an NYISO data set including 1819 buses, 2207 lines, 362 generators and 38 wind farms \cite{liang2022weather}. Note that the 38 wind farms include 33 onshore wind farms that are already built, and 5 offshore farms under construction. The wind power data are calculated from the forecast and {\cblue RT} wind speed, obtained from the NREL WIND Toolkit \cite{nrel_gov,draxl2015wind} {\cblue which provides comprehensive data \textit{and} forecast information with sufficient spatial coverage and resolution.}
The load data are from the NYISO data platform \cite{nyiso}. Due to data availability, we use the load data from August 2018. We use {\cblue RT} wind power and load data of the whole month to {\cblue estimate the covariance matrix as the sample covariance matrix}, calculate leading modes, and pick a windy day in the month to solve the SCUC problem. {\cblue We choose the number of leading modes $K=3$ because from empirical experience the first three modes are enough to explain most variability (more precisely, the three leading eigenvalues of the covariance matrix account for almost all of its Frobenius norm).} The {\cblue DA} and {\cblue RT} VOLL are respectively set at 10,000 \$/MWh and 20,000 \$/MWh, {\cblue to account for the fact that not meeting demand in the {\cblue DA} stage will not necessarily lead to load-shedding in real time.}
The set $\Trt$ includes the hour 6-7 pm, which is a time period with peak load, and we allow commitment decisions to deviate from its deterministic counterpart between 5 pm and 8 pm.

We generate the grids in grid search as follows. First, we fix $\Sigma^{ind} = 3R^{ind}$ and relax constraint \eqref{eq: sub_adversarial_2}. We then allow the stressors $\alpha^{ind}_{kt}$ to be either $R^{ind}$ or $-R^{ind}$ so for each time period $t$, there are a total of 8 (i.e., $2^3$) possible grids. For wind generation $\prt$, we include all 8 grids, while for load $\drt$, we include 2 grids, $[R^d, -R^d, R^d]$ and $[R^d, R^d, -R^d]$, which are usually the most adverse stressors. Note that when $R^{ind}$ is large, it is possible that $\drt$ or $\prt$ becomes negative at one of the grids. If this happens, we adjust the value of the third stressor $\alpha^{ind}_{3t}$ to the largest (if $\alpha^{ind}_{3t} = R^{ind}$) or the smallest (if $\alpha^{ind}_{3t} = -R^{ind}$) value that ensures nonnegativity of $\drt/\prt$. {\cblue Our grid generation strategy is based on the optimal solution of small instances, and it generalizes well to larger instances.}

We run all experiments on a Linux workstation with Intel Xeon processor and 250 GB memory. The programming language is Python v3.8. Optimization problems are solved with Gurobi 10.0.1 \cite{gurobi}. For the master problem, we set the MIP gap to $10^{-3}$, and the time limit to 24 hours. All instances are solved to within 0.5\% optimality gap upon termination.

\subsection{Performance of Risk-Aware Model}\label{ch: results_in_sample}

{\cblue We first study the performance of the risk-aware SCUC assuming that the {\cblue RT} realization is equal to the worst-case outcome.}
Table \ref{tab: in_sample_1} shows the comparison between results from risk-aware and deterministic SCUC problems with different bounds for stressors. $R^d$ takes the values 0.1 and 0.2, and $R^w$ varies from 0.2 to 1.0 with stepsize 0.2. The weight for {\cblue RT consumer exposure} $\rho = 1$. Column {\ccblue ``Save"} shows the savings of the risk-aware SCUC compared with the deterministic SCUC in total costs {\ccblue of DA cost and RT consumer exposure}; ``Deter. cost" reports the total costs of the deterministic SCUC; ``Cost red." equals {\ccblue ``Save"} divided by ``Deter. cost"; ``DA cost diff" is the extra {\cblue DA} cost for carrying out the risk-aware dispatch decisions; {\ccblue ``Consr. exp. diff." equals the RT consumer exposure in the risk-aware SCUC minus its deterministic counterpart}. In addition, Table \ref{tab: in_sample_2} lists the change in total load and wind generation under the selected adverse scenarios. The change in load is calculated as (Total load under adverse scenario - Total expected load)/(Total expected load). The change in wind generation is compared with total expected load, and is calculated as (Total wind generation under adverse scenario - Total expected wind generation)/(Total expected load). Note that for some instances (distinguished by $*$ in the tables) we add 3 cuts at the root node of the branch-and-cut algorithm to bring the optimality gap below 0.5\%. {\cblue In Table \ref{tab: in_sample_2} we also report the optimality gap at the 24-hour time limit \footnotemark and the number of cuts generated. Instances with greater load variations tend to have larger optimality gaps, while the variation in the number of cuts is relatively small.} 

\footnotetext{\cblue We also experimented with a 6-hour time limit, where the optimality gaps for most instances are also below 0.5\%. The optimality gaps become larger (around 1\%) for the last 3 instances, when both $R^d$ and $R^w$ are large. }


\begin{table}[htbp]
\centering
\caption{\small \sc Comparison of Risk-Aware and Deterministic SCUC Problems for $\rho = 1$}\label{tab: in_sample_1}
\begin{tabular}{rrrrrrr}
\hline
\multicolumn{1}{l}{$R^d$} & \multicolumn{1}{l}{$R^w$} & \multicolumn{1}{l}{\ccblue Save}  & \multicolumn{1}{l}{Deter.}     & \multicolumn{1}{l}{Cost}      & \multicolumn{1}{l}{DA cost}  & \multicolumn{1}{l}{\ccblue Consr. exp.}     \\
\multicolumn{1}{l}{}   & \multicolumn{1}{l}{}   & \multicolumn{1}{l}{(k\$)} & \multicolumn{1}{l}{cost (M\$)} & \multicolumn{1}{l}{red. (\%)} & \multicolumn{1}{l}{diff. (\$)} & \multicolumn{1}{l}{{\ccblue diff.} (k\$)} \\ \hline
0.1                    & 0.2                    & 0.00   & 5.37       & 0.00      & 0.00      & {\ccblue 0.00}       \\
0.1                    & 0.4                    & 0.12   & 5.37       & 0.00      & 116.16    & {\ccblue -0.23}       \\
0.1                    & 0.6                    & 0.00   & 5.37       & 0.00      & 0.00      & {\ccblue 0.00}       \\
0.1                    & 0.8                    & 42.29  & 5.41       & 0.78      & 116.16    & {\ccblue -42.41}       \\
0.1                    & 1.0                    & 42.28  & 5.41       & 0.78      & 123.50    & {\ccblue -42.41}       \\
0.2                    & 0.2                    & 114.71 & 5.50       & 2.09      & 116.16    & {\ccblue -114.83}      \\
*0.2                   & 0.4                    & 114.14 & 5.50       & 2.08      & 688.03    & {\ccblue -114.83}      \\
0.2                    & 0.6                    & 115.74 & 5.50       & 2.11      & 1446.22   & {\ccblue -117.18}      \\
0.2                    & 0.8                    & 108.58 & 5.50       & 1.98      & 8130.20   & {\ccblue -116.71}      \\
*0.2                   & 1.0                    & 115.64 & 5.50       & 2.10      & 1072.06   & {\ccblue -116.71}      \\ \hline
\multicolumn{7}{l}{\footnotesize * Instance solved with 3 root cuts.}\\
\end{tabular}
\end{table}

\begin{table}[htbp]
\centering
\caption{\cblue \small \sc Change in Load and Wind Generation (Compared with Total Expected Load) Under Adverse Scenario, Optimality Gap, and Number of Cuts with Risk-Aware SCUC Problem for $\rho = 1$}
\label{tab: in_sample_2}
{\cblue
\begin{tabular}{rrrrrr}
\hline
\multicolumn{1}{l}{$R^d$} & \multicolumn{1}{l}{$R^w$} & \multicolumn{1}{l}{Load}      & \multicolumn{1}{l}{Wind}      & \multicolumn{1}{l}{Opt.}     & \multicolumn{1}{l}{\# Cuts} \\
\multicolumn{1}{l}{}   & \multicolumn{1}{l}{}   & \multicolumn{1}{l}{diff (\%)} & \multicolumn{1}{l}{diff (\%)} & \multicolumn{1}{l}{gap (\%)} & \multicolumn{1}{l}{}        \\ \hline
0.1                    & 0.2                    & 1.18                          & -0.37                         & 0.15                         & 41417                       \\
0.1                    & 0.4                    & 1.18                          & -0.74                         & 0.15                         & 41875                       \\
0.1                    & 0.6                    & 1.18                          & -1.11                         & 0.16                         & 41509                       \\
0.1                    & 0.8                    & 1.18                          & -1.44                         & 0.16                         & 41884                       \\
0.1                    & 1.0                    & 1.18                          & -1.97                         & 0.16                         & 41315                       \\
0.2                    & 0.2                    & 2.36                          & -0.37                         & 0.32                         & 41348                       \\
*0.2                   & 0.4                    & 2.36                          & -0.74                         & 0.33                         & 38988                       \\
0.2                    & 0.6                    & 2.36                          & -1.41                         & 0.33                         & 41162                       \\
0.2                    & 0.8                    & 2.36                          & -1.75                         & 0.47                         & 41238                       \\
*0.2                   & 1.0                    & 2.36                          & -1.97                         & 0.34                         & 38876                       \\ \hline
\multicolumn{4}{l}{\footnotesize * Instance solved with 3 root cuts.}
\end{tabular}}
\end{table}


The risk-aware SCUC reduces total costs for instances with higher variations. When the variation in load and wind generation increases, there tend to be more savings. The change in $R^d$ has a larger impact on cost saving and the {\cblue RT consumer exposure} compared with the change in $R^w$. This is probably because compared with the impact of $R^w$, an increase in $R^d$ leads to a higher change in load, as shown in Table \ref{tab: in_sample_2}. Also, note that the cost saving does not necessarily increase monotonically with $R^w$, as the total costs of both SCUC models increase with more volatility. 

The risk-aware SCUC has a slightly higher {\cblue DA} cost due to the dispatch of additional generation capacity. This is a relatively small addition compared with the cost saving by implementing the risk-aware commitment decisions. 


{\ccblue 
Additionally, we compare the risk-aware SCUC with a two-stage stochastic programming SCUC model (``stochastic SCUC" for short), as shown in Table \ref{tab: in_sample_sto}. Details of the stochastic SCUC implementation are included in Section \ref{ch: sto_scuc} of the Appendix. The column ``Opt. gap" shows the optimality gap of the stochastic SCUC upon termination; ``Save", ``Cost red.", and ``Consr. exp. diff." are similarly defined as in Table \ref{tab: in_sample_1}, except that the deterministic SCUC is replaced by the stochastic SCUC; ``DA cost diff" equals the DA cost of the risk-aware SCUC minus the DA cost of the stochastic SCUC. 

To solve the stochastic SCUC faster and with smaller optimality gaps, we first run it for 24 hours, and then use the obtained solution as a warm start. Despite this approach, many instances still exhibit large optimality gaps. In comparison, all instances of our model are solved without requiring a warm-start solution.

As compared to our approach, the stochastic SCUC results in higher values for both the DA cost and consumer exposure. The increased DA cost is because the stochastic SCUC allocating more generation capacity, to mitigate expected costs due to uncertainty. However, since the stochastic SCUC does not aim to reduce the consumer exposure, and is  designed to hedge against average rather than adverse scenarios, consumer exposure remains high. 


\begin{table}[htbp]
\centering
\caption{\small \sc Comparison of Risk-Aware and Stochastic SCUC Problems for $\rho = 1$}\label{tab: in_sample_sto}
{\ccblue
\begin{tabular}{rrrrrrr}
\hline
\multicolumn{1}{l}{$R^d$} & \multicolumn{1}{l}{$R^w$} & \multicolumn{1}{l}{Opt.}  & \multicolumn{1}{l}{Save}     & \multicolumn{1}{l}{Cost}      & \multicolumn{1}{l}{DA cost}  & \multicolumn{1}{l}{\ccblue Consr. exp.}     \\
\multicolumn{1}{l}{}   & \multicolumn{1}{l}{}   & \multicolumn{1}{l}{gap (\%)} & \multicolumn{1}{l}{(k\$)} & \multicolumn{1}{l}{red. (\%)} & \multicolumn{1}{l}{diff. (k\$)} & \multicolumn{1}{l}{{\ccblue diff.} (k\$)} \\ \hline
0.1 & 0.2 & 3.69 & 243.93 & 4.34 & -185.95 & {\ccblue -57.97} \\
0.1 & 0.4 & 2.22 & 158.57 & 2.87 & -114.88 & {\ccblue -43.69} \\
0.1 & 0.6 & 3.30 & 234.16 & 4.18 & -176.80 & {\ccblue -57.35} \\
0.1 & 0.8 & 2.55 & 190.63 & 3.43 & -133.17 & {\ccblue -57.47} \\
0.1 & 1.0 & 2.32 & 178.97 & 3.22 & -120.88 & {\ccblue -58.09} \\
0.2 & 0.2 & 0.55 & 136.20 & 2.47 & -20.62 & {\ccblue -115.57} \\
*0.2 & 0.4 & 1.00 & 151.53 & 2.74 & -36.65 & {\ccblue -115.57} \\
0.2 & 0.6 & 3.29 & 289.53 & 5.11 & -173.80 & {\ccblue -115.73} \\
0.2 & 0.8 & 0.40 & 121.72 & 2.21 & -5.53 & {\ccblue -116.19} \\
*0.2 & 1.0 & 0.72 & 146.50 & 2.65 & -31.19 & {\ccblue -116.38} \\ \hline
\multicolumn{7}{l}{\footnotesize * Instance solved with 3 root cuts.}\\
\end{tabular}}
\end{table}}

{\cblue 
We now discuss the implication of our method on pricing and market design. Our model aims to obtain a DA schedule that is risk-aware, and we leave the majority of the DA market structure untouched. After getting the schedule, the LMPs can be calculated in the pricing run of the DA market. The generators are paid the LMPs and an uplift payment if needed, and they would have to operate following the outcome of the SCUC, just as in the current practice. 

Figure \ref{fig: fix_price} shows the marginal prices in DA and RT markets during the critical hour with $R^d = 0.2, R^w = 1.0$. Each dot in the plot represents the price at a bus. Note that in our experiment the DA marginal prices are the same for {\ccblue the risk-aware and deterministic} problems, as in DA the risk-aware problem produces electricity with a similar set of generators as its deterministic counterpart. {\ccblue The DA prices of the stochastic SCUC are also similar to those of the other problems.} For RT, risk-aware problem produces prices that are close to DA, while the prices with the deterministic {\ccblue and stochastic models} are both significantly higher and more volatile. }
\begin{figure}[htbp]
  \centering
  \includegraphics[width=0.5\linewidth]{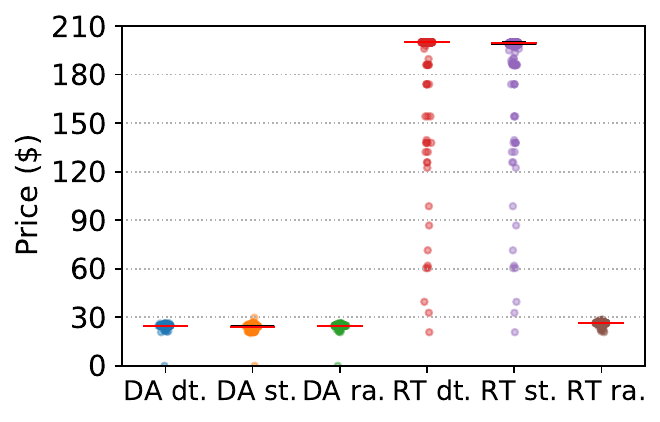}
  \caption{\cblue DA and RT prices with the deterministic{\ccblue, stochastic} and the risk-aware problems (with $R^d = 0.2, R^w = 1.0$). The red line represents the median. In the labels of the horizontal axis, ``dt."{\ccblue, ``st."} and ``ra." are respectively for deterministic{\ccblue, stochastic} and risk-aware.}
  \label{fig: fix_price}
\end{figure}

{\cblue We also observe much higher RT consumer exposure and RT supplier surplus under the deterministic {\ccblue and stochastic models. For the deterministic model, the values} are respectively \$134.50K and {\ccblue \$117.77K}. {\ccblue For the stochastic model, the values are \$134.16K and \$241.03K.} In comparison, the {\ccblue values} for the risk-aware problem are respectively \$17.78K and {\ccblue \$1.36K}.} 

To demonstrate how $\rho$ affects conservatism, Table \ref{tab: in_sample_3} shows the comparison of the two SCUC problems when $\rho = 0.1$. The risk-aware SCUC provides very similar dispatch decisions as the deterministic SCUC when the variations are low. This is because at those volatility levels it is more expensive to mitigate the risk by altering the commitment decisions, and it is not economic to do so given the low weight assigned to RT consumer exposure. {\ccblue This effect of $\rho$ is similar to the effect of the penalty term weight in portfolio optimization.} 

\begin{table}[htbp]
\centering
\caption{\small \sc Comparison of Risk-Aware and Deterministic SCUC Problems for $\rho = 0.1$}
\label{tab: in_sample_3}
\begin{tabular}{rrrrrrr}
\hline
\multicolumn{1}{l}{$R^d$} & \multicolumn{1}{l}{$R^w$} & \multicolumn{1}{l}{Save}  & \multicolumn{1}{l}{Deter.}     & \multicolumn{1}{l}{Cost}      & \multicolumn{1}{l}{{\cblue DA} cost}  & \multicolumn{1}{l}{\ccblue Consr. exp.}     \\
\multicolumn{1}{l}{}   & \multicolumn{1}{l}{}   & \multicolumn{1}{l}{(k\$)} & \multicolumn{1}{l}{cost (M\$)} & \multicolumn{1}{l}{red. (\%)} & \multicolumn{1}{l}{diff. (\$)} & \multicolumn{1}{l}{{\ccblue diff.} (k\$)} \\ \hline
0.1                    & 0.2                    & 0.00                      & 5.37                           & 0.00                          & 0.00                          & {\ccblue 0.00}                          \\
0.1                    & 0.4                    & 0.00                      & 5.37                           & 0.00                          & 0.00                          & {\ccblue 0.00}                           \\
0.1                    & 0.6                    & 0.00                      & 5.37                           & 0.00                          & 0.00                          & {\ccblue 0.00}                           \\
0.1                    & 0.8                    & 0.00                      & 5.41                           & 0.00                          & 0.00                          & {\ccblue 0.00}                          \\
0.1                    & 1                      & 0.00                      & 5.41                           & 0.00                          & 0.00                          & {\ccblue 0.00}                         \\
0.2                    & 0.2                    & 114.66                    & 5.50                           & 2.09                          & 161.75                        & {\ccblue -114.83}                          \\
0.2                    & 0.4                    & 114.22                    & 5.50                           & 2.08                          & 608.60                        & {\ccblue -114.83}                          \\
0.2                    & 0.6                    & 110.11                    & 5.50                           & 2.00                          & 6050.39                       & {\ccblue -116.16}                          \\\hline
\end{tabular}
\end{table}

Finally, note that the risk-aware model becomes more difficult to solve with an increase in the values of $R^d$, $R^w$, and $\rho$. Therefore, to achieve a model that balances computational efficiency and effectiveness, a practitioner needs to make thoughtful parameter selections.

\subsection{Out-of-Sample Tests}
In the previous section, we have shown that the risk-aware model leads to savings under an adverse realization picked by $\dcopfa$. In this section, we test the robustness of the results under different realizations. In other words, we carry out ``out-of-sample" tests to evaluate the dispatch decisions, by first obtaining dispatch decisions from risk-aware and deterministic SCUC models, and then formulating corresponding {\cblue RT} DCOPF problems with simulated random realizations. This process allows us to assess the cost-saving and volatility reduction benefits of the risk-aware model.

More specifically, we randomly sample 100 realizations with two different methods, and evaluate the risk-aware model with both types of random samples. Note that for experiments in this section, the {\cblue RT} DCOPF problem considers 5-minute time intervals in $\Trt$, providing a more realistic setup that mirrors real-life operations. As we will demonstrate shortly, even though the risk-aware dispatch solution is obtained based on 1-hour time interval, it still yields benefits when the {\cblue RT} problem is solved more frequently with smaller time intervals.

Our first experiment checks whether it is beneficial to use the risk-aware dispatch decisions under different adverse realizations. In particular, we perturb the selected adverse stressors $(\forall t\in\Trt, ind \in\{d, w\})~\valpha^{ind*} = (\alpha^{ind*}_{1t}, \alpha^{ind*}_{2t}, \alpha^{ind*}_{3t})$ by randomly sample vectors with a fixed norm, where those vectors are uniformly distributed in a spherical cone centered at $\valpha^{ind*}$ with the cone angle equals $\pi/3$. This perturbation fixes the magnitude of the adverse realization, while allowing the weight distribution to vary among the 3 leading modes to a certain degree. Using those random samples, we evaluate the risk-aware solutions under two different levels of conservatism with $R^d = 0.2, R^w = 0.2$ and $R^d=0.2, R^w = 1.0$. The results are presented respectively in Table \ref{tab:sim_rot_1} and Table \ref{tab:sim_rot_2}.

For both tables, we show the comparison between risk-aware and deterministic results when there are different bounds on stressors. $R^d$ takes the values 0.1 and 0.2, and $R^w$ takes values between 0.2 and 1.4 with stepsize 0.4. Columns ``Save", ``Cost red.", and ``Consr. exp." have similar definitions as in Section \ref{ch: results_in_sample}, except that they are averaged over 100 random samples for out-of-sample evaluation. ``Deter. std" and ``RA. std" are respectively the average standard deviations of total cost under deterministic and risk-aware dispatch solutions.

\begin{table}[htbp]
\centering
\caption{\small \sc Comparison of {\ccblue Risk-Aware and Deterministic} SCUC Solutions Obtained for $R^d = 0.2, R^w = 0.2$, Evaluated with Perturbed Adverse Realizations}
\label{tab:sim_rot_1}
\begin{tabular}{rrrrrrr}
\hline
\multicolumn{1}{l}{$R^d$} & \multicolumn{1}{l}{$R^w$} & \multicolumn{1}{l}{Save}  & \multicolumn{1}{l}{Cost}      & \multicolumn{1}{l}{Consr.}     & \multicolumn{1}{l}{Deter.}    & \multicolumn{1}{l}{RA.}       \\
\multicolumn{1}{l}{}   & \multicolumn{1}{l}{}   & \multicolumn{1}{l}{(k\$)} & \multicolumn{1}{l}{red. (\%)} & \multicolumn{1}{l}{{\cblue exp.} (k\$)} & \multicolumn{1}{l}{std (k\$)} & \multicolumn{1}{l}{std (k\$)} \\ \hline
0.1                    & 0.2                    & -0.04                     & 0.00                          & 7.55                           & 0.46                          & 0.43                          \\
0.1                    & 0.6                    & 2.31                      & 0.04                          & 8.33                           & 3.49                          & 1.91                          \\
0.1                    & 1.0                      & 4.07                      & 0.08                          & 11.16                          & 4.13                          & 3.25                          \\
0.1                    & 1.4                    & 2.97                      & 0.06                          & 16.06                          & 4.30                          & 4.11                          \\
0.2                    & 0.2                    & 17.57                     & 0.32                          & 49.71                          & 20.34                         & 16.29                         \\
0.2                    & 0.6                    & 15.40                     & 0.28                          & 49.51                          & 15.71                         & 13.34                         \\
0.2                    & 1.0                      & 11.86                     & 0.22                          & 54.63                          & 13.62                         & 13.68                         \\
0.2                    & 1.4                    & 7.86                      & 0.14                          & 57.28                          & 13.60                         & 13.53                         \\ \hline
\end{tabular}
\end{table}

\begin{table}[htbp]
\centering
\caption{\small \sc Comparison of {\ccblue Risk-Aware and Deterministic} SCUC Solutions Obtained for $R^d = 0.2, R^w = 1.0$, Evaluated with Perturbed Adverse Realizations}
\label{tab:sim_rot_2}
\begin{tabular}{rrrrrrr}
\hline
\multicolumn{1}{l}{$R^d$} & \multicolumn{1}{l}{$R^w$} & \multicolumn{1}{l}{Save}  & \multicolumn{1}{l}{Cost}      & \multicolumn{1}{l}{Consr.}     & \multicolumn{1}{l}{Deter.}    & \multicolumn{1}{l}{RA.}       \\
\multicolumn{1}{l}{}   & \multicolumn{1}{l}{}   & \multicolumn{1}{l}{(k\$)} & \multicolumn{1}{l}{red. (\%)} & \multicolumn{1}{l}{{\cblue exp.} (k\$)} & \multicolumn{1}{l}{std (k\$)} & \multicolumn{1}{l}{std (k\$)} \\ \hline
0.1                    & 0.2                    & -0.88                     & -0.02                         & 7.43                           & 0.45                          & 0.43                          \\
0.1                    & 0.6                    & 0.70                      & 0.01                          & 7.44                           & 2.59                          & 0.44                          \\
0.1                    & 1.0                      & 4.47                      & 0.08                          & 7.42                           & 3.63                          & 0.43                          \\
0.1                    & 1.4                    & 9.76                      & 0.18                          & 8.26                           & 4.48                          & 1.72                          \\
0.2                    & 0.2                    & 51.33                     & 0.94                          & 15.27                          & 19.48                         & 0.99                          \\
0.2                    & 0.6                    & 43.48                     & 0.80                          & 17.87                          & 15.46                         & 4.72                          \\
0.2                    & 1.0                      & 37.17                     & 0.68                          & 24.81                          & 13.51                         & 8.43                          \\
0.2                    & 1.4                    & 31.31                     & 0.58                          & 31.73                          & 13.85                         & 10.43                         \\ \hline
\end{tabular}
\end{table}

\begin{table}[htbp]
\centering
\caption{\small \ccblue \sc Comparison of Risk-Aware and Stochastic SCUC Solutions Obtained for $R^d = 0.2, R^w = 1.0$, Evaluated with Perturbed Adverse Realizations}
\label{tab:sim_sto}
{\ccblue
\begin{tabular}{rrrrrr}
\hline
\multicolumn{1}{l}{$R^d$} & \multicolumn{1}{l}{$R^w$} & \multicolumn{1}{l}{Save (k\$)}  & \multicolumn{1}{l}{Cost red. (\%)} & \multicolumn{1}{l}{Sto. std (k\$)} \\ \hline
0.1 & 0.2 & 31.91 & 0.59 & 0.64 \\
0.1 & 0.6 & 32.16 & 0.60 & 0.97 \\
0.1 & 1.0 & 34.93 & 0.65 & 3.19 \\
0.1 & 1.4 & 38.56 & 0.71 & 3.74 \\
0.2 & 0.2 & 66.90 & 1.23 & 14.71 \\
0.2 & 0.6 & 62.90 & 1.15 & 12.61 \\
0.2 & 1.0 & 58.45 & 1.07 & 12.19 \\
0.2 & 1.4 & 55.80 & 1.02 & 12.55 \\ \hline
\end{tabular}}
\end{table}

In both Table \ref{tab:sim_rot_1}
and Table \ref{tab:sim_rot_2}, the risk-aware decisions save costs in almost all instances, despite the perturbation on realizations. Generally, the cost saving is larger under higher variations. Yet when there are extremely high variations the cost saving starts to decrease, as the extra capacity dispatched under the risk-aware solutions is not enough to keep LMPs below $\cvoll$ at many buses. We also note that the risk-aware dispatch decisions lead to lower volatility in total cost, as shown by the comparison of the standard deviations.

Compared with Table \ref{tab:sim_rot_1}, Table \ref{tab:sim_rot_2} shows risk-aware solutions lead to higher savings and lower standard deviation in costs at high variation levels. Thus, to achieve more savings, it is important to adjust the values of $R^d$ and $R^w$ accordingly when formulating the risk-aware SCUC problem. For example, if for a certain day the {\cblue RT} wind generation is expected to be very volatile, the practitioner should solve the risk-aware SCUC problem with a higher $R^w$ value.

{\ccblue Table \ref{tab:sim_sto} presents a comparison between risk-aware and stochastic models, which shows that the risk-aware model saves costs in all instances. In general, the stochastic model leads to higher cost volatility compared with the risk-aware model, but lower cost volatility than the deterministic model.}


{\cblue Figure \ref{fig: rotate_prices_surplus_1} shows the marginal prices in DA and RT markets, where each dot represents the price at a bus averaged over all samples and time periods. Similar to our observations with Figure \ref{fig: fix_price}, the risk-aware problem leads to lower and more stable RT prices compared with the {\ccblue other models. Prices of the stochastic model have a lower median compared with those of the deterministic model; however, prices at some buses are much higher than the median for the stochastic model. This is probably because some stochastic scenarios can be overly optimistic at certain buses, leading to commitment decisions that incur high RT costs. Also, since the stochastic SCUC only includes 10 scenarios (as discussed in Section \ref{ch: sto_scuc} of the Appendix), overly optimistic scenarios could be over-represented, exacerbating the issue.} 

Figures \ref{fig: rotate_prices_surplus_2} and \ref{fig: rotate_prices_surplus_3} respectively show the RT consumer exposure (as defined in \eqref{eq:excess}) and producer surplus, with each dot representing the average value of a sample. Again, those values are much lower with the risk-aware problem. {\ccblue Compared with the deterministic SCUC, the stochastic SCUC leads to a slightly lower median for the RT consumer exposure, possibly because the stochastic SCUC commits more DA capacities. On the other hand, likely due to high prices at certain buses, the RT producer surplus median for the stochastic model is the highest among all models.}}

\begin{figure}[htbp]
    \centering
    \subfloat[]{{\includegraphics[width=0.25\textwidth]{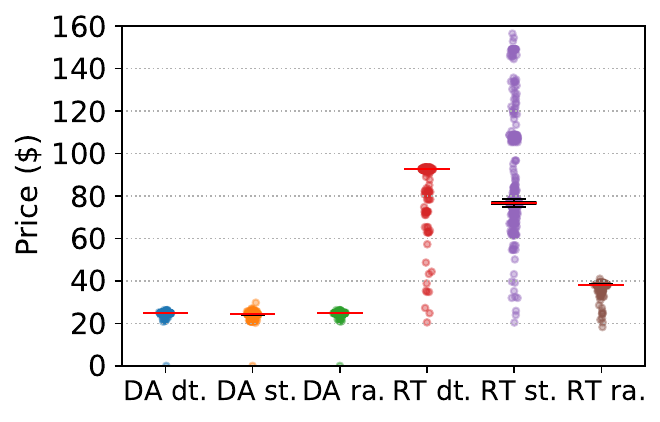}}\label{fig: rotate_prices_surplus_1}}
    \qquad
    \subfloat[]{{\includegraphics[width=0.12\textwidth]{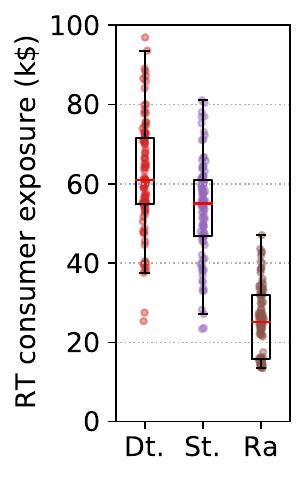}}\label{fig: rotate_prices_surplus_2}}
    \subfloat[]{{\includegraphics[width=0.12\textwidth]{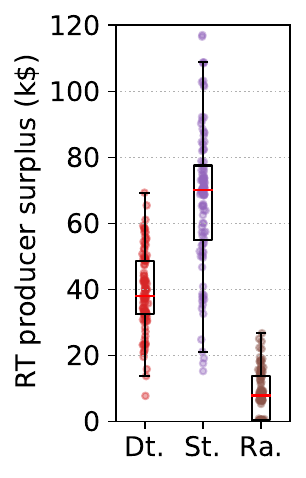}}\label{fig: rotate_prices_surplus_3}}
    \caption{\cblue Evaluated with perturbed adverse realizations for $R^d = 0.2, R^w = 1.0$: (a) DA and RT prices. (b) RT consumer exposure. (c) RT producer surplus. For all figures: The red line represents the median. The height of the boxes in (b) and (c) represents the interquartile range.}%
    \label{fig: rotate_prices_surplus}%
\end{figure}

Our second experiment evaluates the benefit of risk-aware dispatch decisions under less adverse realizations. We generate samples of $\alpha^{ind}_{kt} (\forall k=1,2,3;t\in\Trt, ind\in\{d, w\})$ that follow a uniform distribution in $[-R^{ind}, R^{ind}]$, and calculate corresponding realizations. Note that we truncate the realizations with a 0 lower bound to avoid negative load and wind generation. Again, we evaluate the dispatch solutions obtained with $R^d=0.2,R^w=0.2$ and $R^d=0.2, R^2=1.0$. The results {\ccblue for the comparison between deterministic and risk-aware models} are presented respectively in Table \ref{tab:sim_uniform_1} and Table \ref{tab:sim_uniform_2}. 

\begin{table}[htbp]
\centering
\caption{\small \sc Comparison of {\ccblue Risk-Aware and Deterministic} SCUC Solutions Obtained for $R^d = 0.2, R^w = 0.2$, Evaluated with Uniformly-Distributed Stressors}
\label{tab:sim_uniform_1}
\begin{tabular}{rrrrrrr}
\hline
\multicolumn{1}{l}{$R^d$} & \multicolumn{1}{l}{$R^w$} & \multicolumn{1}{l}{Save}  & \multicolumn{1}{l}{Cost}      & \multicolumn{1}{l}{Consr.}     & \multicolumn{1}{l}{Deter.}    & \multicolumn{1}{l}{RA.}       \\
\multicolumn{1}{l}{}   & \multicolumn{1}{l}{}   & \multicolumn{1}{l}{(k\$)} & \multicolumn{1}{l}{red. (\%)} & \multicolumn{1}{l}{{\cblue exp.} (k\$)} & \multicolumn{1}{l}{std (k\$)} & \multicolumn{1}{l}{std (k\$)} \\ \hline
0.1                    & 0.2                    & -0.09                     & 0.00                          & 2.69                           & 0.62                          & 0.62                          \\
0.1                    & 0.6                    & -0.09                     & 0.00                          & 2.61                           & 0.52                          & 0.52                          \\
0.1                    & 1.0                    & 0.06                      & 0.00                          & 2.60                           & 0.90                          & 0.59                          \\
0.1                    & 1.4                    & 0.33                      & 0.01                          & 2.70                           & 1.38                          & 0.78                          \\
0.2                    & 0.2                    & 0.61                      & 0.01                          & 5.59                           & 2.55                          & 1.22                          \\
0.2                    & 0.6                    & 1.65                      & 0.03                          & 6.19                           & 4.23                          & 2.23                          \\
0.2                    & 1.0                    & 2.68                      & 0.05                          & 7.58                           & 6.28                          & 3.95                          \\
0.2                    & 1.4                    & 2.30                      & 0.04                          & 8.23                           & 6.50                          & 4.95                          \\
0.3                    & 0.2                    & 6.11                      & 0.11                          & 19.62                          & 15.17                         & 12.89                         \\
0.3                    & 0.6                    & 6.50                      & 0.12                          & 20.53                          & 12.85                         & 11.71                         \\
0.3                    & 1.0                    & 6.13                      & 0.11                          & 21.87                          & 14.51                         & 11.34                         \\
0.3                    & 1.4                    & 5.45                      & 0.10                          & 22.45                          & 12.09                         & 10.68                         \\ \hline
\end{tabular}
\end{table}

\begin{table}[htbp]
\centering
\caption{\small \sc Comparison of {\ccblue Risk-Aware and Deterministic} SCUC Solutions Obtained for $R^d = 0.2, R^w = 1.0$, Evaluated with Uniformly-Distributed Stressors}
\label{tab:sim_uniform_2}
\begin{tabular}{rrrrrrr}
\hline
\multicolumn{1}{l}{$R^d$} & \multicolumn{1}{l}{$R^w$} & \multicolumn{1}{l}{Save}  & \multicolumn{1}{l}{Cost}      & \multicolumn{1}{l}{Consr.}     & \multicolumn{1}{l}{Deter.}    & \multicolumn{1}{l}{RA.}       \\
\multicolumn{1}{l}{}   & \multicolumn{1}{l}{}   & \multicolumn{1}{l}{(k\$)} & \multicolumn{1}{l}{red. (\%)} & \multicolumn{1}{l}{{\cblue exp.} (k\$)} & \multicolumn{1}{l}{std (k\$)} & \multicolumn{1}{l}{std (k\$)} \\ \hline
0.1                    & 0.2                    & -0.88                     & -0.02                         & 2.52                           & 0.62                          & 0.59                          \\
0.1                    & 0.6                    & -0.88                     & -0.02                         & 2.44                           & 0.52                          & 0.50                          \\
0.1                    & 1.0                    & -0.74                     & -0.01                         & 2.45                           & 0.90                          & 0.57                          \\
0.1                    & 1.4                    & -0.36                     & -0.01                         & 2.44                           & 1.38                          & 0.58                          \\
0.2                    & 0.2                    & 0.08                      & 0.00                          & 5.15                           & 2.55                          & 1.19                          \\
0.2                    & 0.6                    & 1.65                      & 0.03                          & 5.23                           & 4.23                          & 1.19                          \\
0.2                    & 1.0                    & 4.03                      & 0.07                          & 5.28                           & 6.28                          & 1.51                          \\
0.2                    & 1.4                    & 4.34                      & 0.08                          & 5.23                           & 6.50                          & 1.60                          \\
0.3                    & 0.2                    & 16.66                     & 0.31                          & 8.11                           & 15.17                         & 2.05                          \\
0.3                    & 0.6                    & 17.60                     & 0.33                          & 8.48                           & 12.85                         & 2.21                          \\
0.3                    & 1.0                    & 16.69                     & 0.31                          & 10.35                          & 14.51                         & 4.76                          \\
0.3                    & 1.4                    & 15.79                     & 0.29                          & 11.16                          & 12.09                         & 6.14                          \\ \hline
\end{tabular}
\end{table}

\begin{table}[htbp]
\centering
\caption{\small \ccblue \sc Comparison of Risk-Aware and Stochastic SCUC Solutions Obtained for $R^d = 0.2, R^w = 1.0$, Evaluated with Uniformly-Distributed Stressors}
\label{tab:sim_uniform_sto}
{\ccblue
\begin{tabular}{rrrrrr}
\hline
\multicolumn{1}{l}{$R^d$} & \multicolumn{1}{l}{$R^w$} & \multicolumn{1}{l}{Save (k\$)}  & \multicolumn{1}{l}{Cost red. (\%)} & \multicolumn{1}{l}{Sto. std (k\$)} \\ \hline
0.1 & 0.2 & 30.73 & 0.57 & 0.78 \\
0.1 & 0.6 & 30.67 & 0.57 & 0.62 \\
0.1 & 1.0 & 30.70 & 0.57 & 0.67 \\
0.1 & 1.4 & 30.85 & 0.57 & 1.08 \\
0.2 & 0.2 & 31.58 & 0.58 & 1.61 \\
0.2 & 0.6 & 32.44 & 0.60 & 3.21 \\
0.2 & 1.0 & 33.75 & 0.62 & 3.80 \\
0.2 & 1.4 & 34.23 & 0.63 & 5.09 \\
0.3 & 0.2 & 43.20 & 0.80 & 11.62 \\
0.3 & 0.6 & 43.79 & 0.81 & 11.67 \\
0.3 & 1.0 & 43.51 & 0.80 & 11.32 \\
0.3 & 1.4 & 43.13 & 0.80 & 13.21 \\ \hline
\end{tabular}}
\end{table}

Even with less adverse realizations, the risk-aware dispatch solutions still save costs at higher variation levels. The standard deviation of total costs is also smaller under risk-aware solutions, indicating a more reliable dispatch schedule. Due to the dispatch of additional capacity, the risk-aware solution is more costly at very low variation levels, yet such cost is relatively small compared with potential savings when there are higher variations.

{\ccblue Table \ref{tab:sim_uniform_sto} presents the comparison between risk-aware and stochastic models under less adverse realization. The stochastic model again leads to higher total costs and higher standard deviations. }

{\cblue For less adverse realizations, the results for marginal prices, RT consumer exposure, and RT producer surplus are similar to Figure \ref{fig: rotate_prices_surplus}, and thus we omit them here.}


{\cblue Note that in practice the {\cblue RT} market is cleared in a rolling horizon manner, where at each {\cblue RT} time period generation output is decided based on the current net load condition. Our out-of-sample test serves as a bound case for this rolling horizon market clearing procedure.}

\section{Conclusion}\label{sec: conclusion}
In this work, we enhance the SCUC computation to better handle load and wind generation volatility, which reduces the {\cblue RT consumer exposure} due to {\cblue RT} price spike. Our method features a data-driven PCA-based uncertainty set, which models the locational correlation in uncertain data. We develop cutting planes and heuristics to solve the non-convex optimization model efficiently. Validated through the extensive case study on an NYISO data set, our approach effectively reduces total costs and cost volatility under adverse scenarios. Notably, these benefits are observed across various levels of variation, and are achieved without substantial expenses for dispatch. 
\appendix
\subsection{Example: Binding Ramping Constraint and LMP}\label{ch: ramp_example}
We provide an example to demonstrate that when $\vy$ changes to $\vy'$ and when a ramping constraint is binding, it is possible that the LMP at hour $\hat{t}$ decreases even if $\mc{I}_{0\hat{t}} \subseteq \mc{I}'_{0\hat{t}}$, i.e., all generators were off under $\vy$ remain off when switching to $\vy'$. 

Consider a system with uncongested network, three thermal generators $g\in\{1,2,3\}$ and two hours $t\in\{1,2\}$. Let the {\cblue RT} load be 59 for both hours and $\pmin$ equals 0 for all generators. Also, let the generator capacity $\mathbf{P}^\max = [50, 10, 50]$ and {\cblue RT} ramping limit $\mathbf{M}^{\rm{RT}}=[\infty,\infty, 1]$. Assume that all generators have no startup or shutdown cost and that they have constant cost for each MWh of power output, which equals 1, 2, and 3 respectively for the three generators. Initially, let all generators be committed in both hours, i.e., $\mc{I}_{0t} = \emptyset, \forall t\in\{1,2\}$. Then the vector of optimal production decisions ($p_{gt}$) is $\mathbf{p} = [50, 9, 0; 50, 9, 0]$ and the LMP $\pmb \lambda = [2, 2]$.

Now we switch from $\vy$ to $\vy'$ by turning off $g=1$ at hour $t=1$. Then the updated production decisions $\mathbf{p}'=[0,10,49;11,0,48]$. Note that $p'_{32} = 48$ because the ramping limit of generator 3 is 1. Since generator 1 has no ramping limit and $p'_{12}\in (0,50)$, it is the marginal generator and sets the LMP in hour 2 and thus $\lambda'_2 = 1$. Therefore, the LMP at $t=2$ decreases even if $\mc{I}_{02} \subseteq \mc{I}'_{02} = \emptyset$.
{\ccblue
\subsection{Two-Stage Stochastic Programming SCUC Model}\label{ch: sto_scuc}
The two-stage stochastic programming SCUC model \citep{zhang2018conditional, asensio2015stochastic} used for benchmarking in numerical experiments of Section \ref{sec: numerical} minimizes the sum of the expected total cost and the weighted term for CVaR of the total cost:
\begin{subequations}\label{sto_bench}
\begin{alignat}{4}
\min~~& \frac{1}{|\mc{S}|} \sum_{s\in\mc{S}} c_s + \rho \left(\var + \frac{1}{1-\beta} \frac{1}{|\mc{S}|}\sum_{s\in\mc{S}} \eta_s\right)\hspace{-12cm}\label{sto_bench_0}\\
\st~~& \eta_s\geq c_s - \var,~~\forall s\in\mc{S}\label{sto_bench_1}\\
& c_s = \sum_{t\in\Tda}\Big(\sum_{g\in\Gt} \left(h_{sgt} + \cstart v_{gt} + \cdown w_{gt}\right)\nonumber\hspace{-12cm}\\& + \sum_{i\in\N} \cvollrt \punmets\Big),~~\forall s\in\mc{S}\label{sto_bench_2}\\
& h_{sgt} \geq C^1_{og} p_{sgt} + C^0_{og}y_{gt} ,~~\forall s\in\mc{S}, o\in\mc{O}, g\in\Gt \label{sto_bench_3}\\
&\sum_{g\in\mc{G}_i} p_{sgt} + p^{Unmet}_{sit}+ \sum_{(j,i)\in\mc{L}}f_{sjit} \hspace{-12cm}\\& - \sum_{(i,j)\in\mc{L}} f_{sijt} = D_{sit} ,~~\forall s\in\mc{S}, i\in\mc{N}, t\in\mc{T}\\
& f_{sijt} = B_{ij} \left(\theta_{sit} - \theta_{sjt}\right),~~\forall s\in\mc{S}, (i, j)\in\mc{L}, t\in\mc{T}\\
& f_{sijt} \leq P_{ij}^{Trans} ,~~\forall s\in\mc{S}, (i, j)\in\mc{L}, t\in\mc{T} \\
& f_{sijt}\geq - P_{ij}^{Trans},~~ \forall s\in\mc{S}, (i, j)\in\mc{L}, t\in\mc{T}\\
&p_{sgt} \leq P^\max_g y_{gt} ,~~\forall s\in\mc{S}, g\in\mc{G}, t\in\mc{T} \\
& p_{sgt} \geq P^\min_g y_{gt} ,~~ \forall  s\in\mc{S}, g\in\mc{G}, t\in\mc{T}\\
& p_{sgt} - p_{sg, t-1} \leq M_g y_{g, t-1} + P^\min_g v_{gt},~~\nonumber\\ &\hspace{2.3cm}\forall s\in\mc{S}, g\in\Gt, t\in\mc{T}\sm\{1\}\\
& p_{sg,t-1} - p_{sgt} \leq M_g y_{gt} + P^\min_g w_{gt},~~\nonumber\\ &\hspace{2.25cm}\forall s\in\mc{S}, g\in\Gt, t\in\mc{T}\sm\{1\}\\
& p_{sgt} + \pcurtails = \pmaxs  ,~~ \forall s\in\mc{S}, g\in\Gw, t\in\mc{T}\\
& u_{gt}-w_{gt} = y_{gt} - y_{g, t-1}  ,~~\forall g\in\mc{G}, t\in\mc{T}\sm\{1\}\label{sto_bench_13}\\
& u_{gt}, w_{gt}, y_{gt} \in \{0,1\} ,~~ \forall g\in\mc{G}, t\in\mc{T}\\
& p^{Unmet}_{sit}, \eta_s \geq 0 ,~~\forall s\in\mc{S}, i\in\mc{N}, t\in\mc{T}\label{sto_bench_15}\\
& \pcurtails \geq 0,~~ \forall s\in\mc{S}, g\in\Gw, t\in\mc{T}.
\end{alignat}
\end{subequations}

The first term of the objective function \eqref{sto_bench_0} is the expected total cost, where $\mc{S}$ is the set of scenarios (we assume each scenario is equally likely) and $c_s$ is the total cost of scenario $s\in\mc{S}$. The second term of \eqref{sto_bench_0} is the CVaR of total cost multiplied by a weight $\rho$. Inside the CVaR term, $\beta$ is the confidence level and $\eta_s$ is an auxiliary variable which equals to $(c_s - \var)^+$, with $\var\in\mb{R}$ being a free variable. \eqref{sto_bench_1} and $\eta_s\geq 0$ enforce $\eta_s = (c_s - \var)^+$. \eqref{sto_bench_3} - \eqref{sto_bench_13} are SCUC constraints. 

In our experiment, we set $\rho = 0$ and $\beta = 0.9$. The RT demand $D_{sit}$ and wind production $\pmaxs$ are randomly sampled from the uncertainty sets $\Omega$, with the weights $\alpha^{ind}_{kt} (\forall k=1,2,3;t\in\Trt, ind\in\{d, w\})$ following a uniform distribution in $[-R^{ind}, R^{ind}]$. We use 10 samples in the experiment. 

As shown in Table \ref{tab: in_sample_sto}, the stochastic SCUC is difficult to solve even with only 10 sampled scenarios. In addition to the warm-start solution method as described in Section \ref{sec: numerical}, we also experimented with the L-shaped method \cite{van1969shaped}, which is a classical algorithm for two-stage stochastic programs, via a branch-and-cut scheme. The L-shaped method struggles to find good feasible solutions, and its performance is occasionally worse than direct solving for all instances.
}






\section*{Acknowledgment}
{\ccblue The authors would like to thank Clemson University for the allotment of compute time on the Palmetto Cluster.}

\ifCLASSOPTIONcaptionsoff
  \newpage
\fi



%
{\footnotesize
\bibliographystyle{IEEEtran}
\bibliography{references}}



%








\end{document}